\documentclass[11pt,leqno,letterpaper]{article}

\usepackage[svgnames]{xcolor}
\usepackage{setspace}
\usepackage{abstract}
\usepackage[normalem]{ulem}
\usepackage{mathtools}
\usepackage[abbrev,lite,nobysame]{amsrefs}
\usepackage{amsthm,amsfonts}
\usepackage{thmtools}
\usepackage{csquotes}
\usepackage{enumerate}
\usepackage{bbm}
\usepackage{dsfont}
\usepackage{marvosym}
\usepackage{tikz}
\usetikzlibrary{patterns}

\usepackage[normalem]{ulem}
\newcommand{\stkout}[1]{\ifmmode\text{\sout{\ensuremath{#1}}}\else\sout{#1}\fi}

\usepackage[looser]{newtxtext}
\usepackage[bigdelims,varvw]{newtxmath}
\usepackage[protrusion=true, tracking=true, kerning=true,spacing=true]{microtype}
\usepackage{enumitem}

\usepackage[colorlinks=true,linkcolor=NavyBlue,
citecolor=NavyBlue,urlcolor=NavyBlue,bookmarksdepth=3,hypertexnames=false]{hyperref}
\usepackage{cleveref}
\crefname{equation}{}{}

\numberwithin{equation}{section}

\theoremstyle{plain}
\newtheorem{theorem}{Theorem}[section]
\newtheorem{lemma}[theorem]{Lemma}
\newtheorem{corollary}[theorem]{Corollary}
\newtheorem{proposition}[theorem]{Proposition}
\newtheorem{example}[theorem]{Example}

\theoremstyle{remark}
\newtheorem{remark}[theorem]{Remark}
\theoremstyle{definition}
\newtheorem{definition}[theorem]{Definition}

\renewcommand{\leq}{\leqslant}
\renewcommand{\le}{\leqslant}
\renewcommand{\ge}{\geqslant}
\renewcommand{\geq}{\geqslant}
\renewcommand{\mathcal}{\mathscr}
\renewcommand{\tilde}{\widetilde}

\newcommand\tfootnote[1]{%
	\begingroup
	\renewcommand\thefootnote{}\footnote{#1}%
	\addtocounter{footnote}{-1}%
	\endgroup
}

\font\bosy=cmbsy10
\def\conc{\hbox{\bosy \char '175}}

\newcommand\ICONT[1]{\mathop{\displaystyle \mathop{\conc}_{#1}}}
\newcommand{\sct}{\scriptstyle}
\def\di#1#2{\substack{\sct#1 \\ \noalign{\vskip-2mm}\\ \scriptstyle#2}}

\begin{document}\linespread{1.05}\selectfont
	\date{}

	\author{Sorin~Micu \protect\footnote{ Department of Mathematics, University of Craiova, 13, Al.I. Cuza Street, Craiova, 200585 and Gheorghe Mihoc-Caius Iacob Institute of Mathematical Statistics and Applied Mathematics of the Romanian Academy, Calea 13 Septembrie, No. 13, Bucharest, 050711, Romania,  e-mail: \href{mailto:sd_micu@yahoo.com}{sd$\_$micu@yahoo.com}} \and Ionel~Roven\c ta \protect\footnote{Department of Mathematics, University of Craiova, Craiova, e-mail: \href{mailto:ionelroventa@yahoo.com}{ionelroventa@yahoo.com}} \and Marius~Tucsnak  \protect\footnote{Institut de Math\'ematiques de Bordeaux, UMR 5251, Universit\'{e} de Bordeaux/Bordeaux INP/CNRS, 351 Cours de la Lib\'eration - F 33 405 TALENCE, France, and Institut Universitaire de France (IUF), e-mail: \href{mailto:marius.tucsnak@u-bordeaux.fr}{marius.tucsnak@u-bordeaux.fr}}}

	\title{Bergman-space regularity for the heat equation with white-noise boundary forcing}
	
	\maketitle
	
	\tfootnote{{{\bf MSC2020}:  	35R60, 93B03, 60H15} {\bf Keywords}: heat equation, white noise boundary inputs, Bergman spaces, reachable space, operator semigroups, well posed linear control systems}

\begin{abstract}
We introduce a Bergman-space framework for the study of boundary-forced heat equations and show that, in the one-dimensional case, boundary white noise gives rise to a sharp holomorphic regularity phenomenon. More precisely, we consider the heat equation on a bounded interval with Dirichlet or Neumann boundary conditions driven by independent white noises at the endpoints, and we prove that for every positive time the corresponding state extends holomorphically to a rhombus in the complex plane having the original interval as one of its diagonals. Moreover, the resulting process admits a continuous version with values in a scale of weighted Bergman spaces on that rhombus, depending on two parameters $\delta\in(0,1)$ and $\Theta\in\left(0,\frac{\pi}{4}\right)$.

To our knowledge, this is the first systematic use of Bergman spaces as state spaces for parabolic equations with stochastic boundary forcing. We also prove that the result is optimal, in the sense that the conclusion fails at the critical values $\delta=0$ and $\Theta=\frac{\pi}{4}$.
\end{abstract}

\newpage
	
\setcounter{tocdepth}{3}	
\tableofcontents

\section{Introduction}\label{true_introduction}

This paper proposes a new functional-analytic point of view on boundary-forced heat equations, based on Bergman spaces on complex domains naturally associated with the spatial interval.  
Our main observation is that, for the one-dimensional heat equation driven by independent white noises at the endpoints, the appropriate state-space description  is not merely Sobolev in nature: it is in fact holomorphic.  
In this sense, the present work builds a bridge between areas which are usually treated separately, namely the analysis of parabolic equations with boundary inputs, the theory of Bergman spaces, and stochastic evolution equations.

More precisely, we show that when the heat equation on a bounded interval is driven at the boundary by white noise, the resulting stochastic evolution still exhibits a strong complex-analytic regularity.  
For every positive time, the solution extends holomorphically to a rhombus in the complex plane having the underlying interval as one of its diagonals, and its trajectories are continuous in time with values in suitable weighted Bergman spaces on that rhombus.  
Thus, although boundary white-noise forcing is usually treated in rough Sobolev-type spaces, the dynamics of the one-dimensional heat equation in fact selects a much finer and intrinsically complex-analytic state space.

The novelty of our results is therefore twofold.  
On the one hand, they provide a sharp regularity statement for stochastic heat equations with noisy boundary data.  
On the other hand, and more importantly from the analytic point of view, they show that Bergman spaces arise naturally as state spaces for boundary-forced parabolic equations with stochastic inputs.  
To our knowledge, these are the first results establishing such a connection in a systematic way.

We informally describe below our main result in the case of Dirichlet boundary conditions, whereas the precise statements for both Dirichlet and Neumann boundary conditions are provided in \Cref{sec_main_not}. 

Our main example is the system:
\begin{equation}\label{PROBHEAT1_stoch}
   \left\{ \begin{array}{lr}\dfrac{\partial \psi}{\partial t}(t,x)
   = \dfrac{\partial^2 \psi}{\partial x^2}(t,x)& \qquad(t\geqslant 0,\
   x\in (0,\pi)),\\ \ &\ \\ \psi(t,0)=  {\dot W}^0_t,
   \ \ \psi\left(t,\pi\right)=  {\dot W}_t^\pi & t\in[ 0,\infty),\\ \ &\ \\ \psi(0,x)
   = \psi_0(x)\ \ \ \ \ \ & x\in \left(0,\pi\right). \end{array}\right.
\end{equation}
In the above equations $\psi$ stands for the state trajectory of the system and $\{W^i_t ,\ t \geqslant 0\}$, with $i \in \{ 1, 2\}$, are independent standard real Wiener processes.

Classical well-posedness results for \eqref{PROBHEAT1_stoch}, going back to Da Prato and Zabczyk \cite{DAP_Z_Main}, are formulated in rough spaces, such as Sobolev spaces of negative order.  
From that perspective, stochastic boundary forcing appears to destroy the strong spatial regularity properties usually associated with parabolic equations.  
Our purpose is to show that, at least in one space dimension, this picture is incomplete: beyond the rough Sobolev framework, the evolution actually takes place in a much finer scale of spaces of Bergman type.

Related boundary-noise problems on the half-line were analyzed in Al\`os and Bonaccorsi \cites{MR1953738,MR1899108} and in Brze\'zniak et al.\ \cite{MR3315663}, while Goldys and Peszat \cite{MR4561684} treated more general parabolic equations on bounded domains.  
A common feature of these works is that well posedness is obtained in rough state spaces, such as negative-order Sobolev spaces or weighted $L^p$ spaces.  
In contrast, the present work identifies weighted Bergman spaces on suitable rhombi as the natural state spaces for the one-dimensional heat equation with stochastic boundary forcing.

The starting point of our analysis is a deterministic analogy.  
Consider the boundary-controlled heat equation
\begin{equation}\label{PROBHEAT1}
   \left\{ \begin{array}{lr}\dfrac{\partial z}{\partial t}(t,x)
   = \dfrac{\partial^2 z}{\partial x^2}(t,x)& t\geqslant 0,\
   x\in (0,\pi),\\ \ &\ \\ z(t,0)=  u_0(t),
   \ \ z\left(t,\pi\right)=  u_\pi(t) & t\in[ 0,\infty),\\ \ &\ \\ 
   z(0,x)
= 0\ \ \ \ \ \ & x\in \left(0,\pi\right), \end{array}\right.
\end{equation}
which models the heat propagation in a rod
of length $\pi$, controlled by prescribing the temperature at both ends.
It is well known, see, for instance, \cite[Theorem 10.4.1]{MR747979}, that for all $u_0,\ u_\pi\in L^2[0,\infty)$ and every $\tau \geqslant 0$ the map
$x\mapsto z(\tau,x)$ extends holomorphically to a square in the complex plane having $[0,\pi]$ as one of its diagonals.  

More recent results by Ervedoza, Le Balc'h and Tucsnak \cite{MR4474841} and by Hartmann and Orsoni \cite{hartmann2021separation} showed that this holomorphic extension phenomenon is not merely a byproduct of parabolic smoothing, but in fact provides the correct state-space description of the deterministic system.  
More precisely, let $D_\frac{\pi}{4}$ be the square in the complex plane having the segment $[0,\pi]$ as one of its diagonals (the apparently strange notation $D_\frac{\pi}{4}$ will be explained below).  
Consider the Bergman space $A^2(D_{\frac{\pi}{4}})$, which is formed by the functions holomorphic and square integrable on $D_{\frac{\pi}{4}}$. Then, combining results from
\cite{MR4474841} and  \cite{hartmann2021separation}, it follows that $A^2(D_{\frac{\pi}{4}})$ is the smallest Hilbert space such that for all $L^2$ inputs $u_0$, $u_\pi$ the solution of \eqref{PROBHEAT1} is continuous in time with values in that space.

These deterministic results suggest a natural question:  
\emph{does this Bergman-space picture survive after passing from deterministic controls to stochastic boundary inputs?}  
Equivalently, can one still describe the trajectories of \eqref{PROBHEAT1_stoch} in terms of canonical holomorphic state spaces, and if so, what is the sharp scale of Bergman spaces selected by the noise?

In this paper we give a positive answer, which also reveals a new rigidity phenomenon.  
In the stochastic case, the relevant holomorphic space is no longer the unweighted Bergman space on the maximal square $D_\frac{\pi}{4}$, but rather a scale, depending on two parameters, of weighted Bergman spaces on smaller rhombi.

To state this more concretely, let $\Theta\in(0,\pi/4]$ and let $D_\Theta$ be the rhombus 
in the complex plane having $[0,\pi]$ as one of its diagonals and an angle of measure $2\Theta$
at its vertex $s=0$, see Figure \ref{fig_unic} below.  
For $\delta\geqslant 0$ we introduce the weight
\begin{equation}\label{def_rho_delta}
 \rho_\delta(s)=|s|^\delta \left|\pi-s\right|^\delta \qquad\qquad(\delta\geqslant 0,\, \, s\in D_\Theta),    
\end{equation}
and we denote by $A_\delta^2(D_\Theta)$  the corresponding weighted Bergman space.  

We defer the precise definitions to the next section and state here the main conclusion for the system \eqref{PROBHEAT1_stoch} in the informal form:

\medskip

\noindent
\emph{For every $\Theta\in\left(0,\frac{\pi}{4}\right)$ and every $\delta\in(0,1)$, the solution of
\eqref{PROBHEAT1_stoch} admits a version continuous in time with values in $A_\delta^2(D_\Theta)$.  
Moreover, this result is sharp: in general one cannot take $\delta=0$, nor can one take the critical angle $\Theta=\frac{\pi}{4}$.}

\medskip

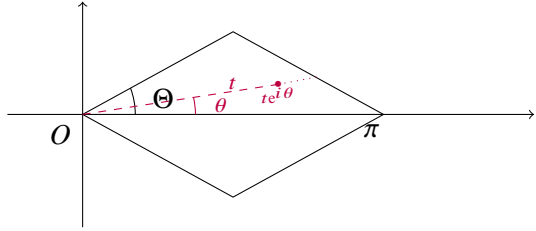
\begin{figure}[h!]
\centering
\begin{tikzpicture}
\draw[->] (-1,0) -- (6,0);
\draw [->] (0,-1.5) -- (0,1.5);
\draw (0, 0) node[below left] {$O$};
\draw (3.6, 0) node[below right] {$\pi$};
%
\draw (0.7,0) arc (0:20.5:1) ;
\draw (0.8,-0.05) node[above right] {$\Theta$};
\draw (0,0) -- (2, 1.1) -- (4, 0) -- (2, -1.1) -- (0, 0) ;
\draw [purple, dashed] (0,0) -- (2.6,0.4);
\draw [purple] (2.6,0.4) node {$\scriptscriptstyle{\bullet}$};
\draw [purple] (2.6,0.5) node[below] {$\scriptscriptstyle{t {\rm e}^{i \theta}}$};
\draw [purple] (1.5,0) arc (0:13:1) ;
\draw [purple] (1.6,-0.1) node [above right] {$\scriptstyle{\theta}$};
\draw [purple] (2,0.2) node [above] {$\scriptstyle{t}$};
\draw [purple, dotted] (2.6,0.4) --  (3.12, 1.6/26*3.12+8/26);
\end{tikzpicture}
\caption{Picture of $D_\Theta$}
\label{fig_unic}
\end{figure}

The above result shows that the appropriate regularity theory for \eqref{PROBHEAT1_stoch} is not merely Sobolev, but genuinely complex-analytic.  
In particular, it singles out a sharp scale of weighted Bergman spaces as the natural state spaces for the stochastic evolution.  
This point of view appears to be new even for the one-dimensional heat equation, and it reveals an unexpected robustness of the holomorphic structure previously observed in deterministic reachable-space theory.

The proof combines recent advances on reachability theory for boundary control systems with tools from Bergman-space theory, thereby extending to the stochastic setting the reachable-space perspective previously developed for deterministic heat equations. Some of the intermediate steps of our main proofs yield results which might be of independent interest,
presenting similarities with recent contributions in  Ervedoza and Tendani  \cite{ervedoza2025reachablespaceparabolicequations}. We think in particular of \Cref{th_hol_tilde} and \Cref{th_hol_tilde_neu}, where we show that the one-dimensional heat equation, with Dirichlet or Neumann homogeneous boundary conditions, is well posed in spaces of Bergman type on the rhombus $D_\Theta$ introduced above.


\section{Notation and roadmap to the main results}\label{sec_main_not}
\subsection{Notation and statement of the main results}

We begin by formalizing our notation for the Bergman spaces, more informally introduced in the previous section. Firstly, given $\Theta\in \left(0,\frac{\pi}{2}\right)$, we define
\begin{equation}\label{rtheta}
D_\Theta
:=
\left\{ 
\lambda \in \mathbb{C}; 
\lambda =  t {\rm e}^{i \theta}
\text{ with }
| \theta | < \Theta 
\text{ and }
t \in 
\left[
0, 
r_\theta
:=
\frac{\pi \sin(| \Theta|)}
{\sin( | \theta| + | \Theta|)}
\right]
\right\},
\end{equation}
(see Figure \ref{fig_unic}). 
In other words, \(D_\Theta\) is the rhombus of
vertices
\[
w_1=0,\ w_2=\frac{\pi}{2}\left(1-i\tan\Theta\right),\ w_3=\pi,\ w_4=\frac{\pi}{2}\left(1+i\tan\Theta\right). 
\]

Remark that for $\Theta=\frac{\pi}{4}$ we retrieve the square $D_\frac{\pi}{4}$ already introduced above.
For every $\Theta\in \left(0,\frac{\pi}{2}\right)$ and $\delta\geqslant 0$ we denote by
$A^2_\delta \left(D_\Theta\right)$ the weighted Bergman space given by 
\begin{equation}\label{fara_rho}
A^2_\delta\left(D_\Theta\right)
:=
L^2_\delta\left(D_\Theta\right) 
\cap 
{\rm Hol}(  D_\Theta),
\end{equation}
where ${\rm Hol}(  D_\Theta)$ is the space of holomorphic function on $D_\Theta$ and
\begin{equation*}
L_\delta^2 \left(D_\Theta\right)=\left\{g\in L^2_{\rm loc} \left(D_\Theta\right)\ \ |\ \ 
\int_{D_\Theta} | g(s)|^2 \,\rho_\delta(s) \, {\rm d} A(s) <\infty\right\},
\end{equation*}
where the weight $\rho_\delta$ has been introduced in \eqref{def_rho_delta}. Here and in the sequel we write ${\rm d}A(s)$ for the two-dimensional Lebesgue measure. 

\begin{remark}\label{rem_ident}
 In the whole remaining part of this work, we systematically identify a function holomorphic on the rhombus $D_\Theta$ with its trace on the interval $(0,\pi)$. Consequently, the assertion $f\in A_\delta^2(D_\Theta)$
 should be understood: $f$ is an analytic function on $(0,\pi)$ which admits a holomorphic extension  on $D_\Theta$, still denoted by $f$, with 
 \[
 \int_{D_\Theta} |  f(s)|^2 \,\rho_\delta(s) \, {\rm d} A(s) < \infty.
 \]
 Note that $A^2_\delta\left(D_\Theta\right)$ becomes a Hilbert space when endowed with the norm 
\[
\langle f,g \rangle_{A_\delta^2(D_\Theta)} =\int_{D_\Theta}  f(s) \overline{g(s)} \rho_\delta(s) \, {\rm d} A(s),
\]
For $\delta=0$ this space is simply denoted by $A^2 \left(D_\Theta\right)$.
\end{remark}

We also introduce, for every $\Theta\in \left(0,\frac{\pi}{2}\right)$ and every $\delta > 0$, the function space
\begin{equation}\label{Sob_Berg}
A_\delta^{1,2}(D_\Theta)=\left\{f\in A_\delta^2(D_\Theta)\ \ |\ \ f'\in A_\delta^2(D_\Theta)\right\},   
\end{equation}
which, when endowed with the inner product
\[
\langle f,g \rangle_{A_\delta^{1,2}(D_\Theta)} =\int_{D_\Theta} \left[ f(s) \overline{g(s)} +  f'(s) \overline{g'(s)}\right]\rho_\delta(s)\, {\rm d} A(s),
\]
is a Hilbert space. For $\delta=0$ this space is simply denoted by $A^{1,2}(D_\Theta)$.


Before precisely stating our main results, we note that a concept of mild solution with values in the negative order Sobolev space $W^{-1,2}(0,\pi)$ of  \eqref{PROBHEAT1_stoch} has been introduced in  \cite{DAP_Z_Main}, where the existence, uniqueness and continuity with respect to time of such solutions has been proved. Moreover, results from  \cites{MR1953738,MR1899108,MR3315663} indicate that similar results hold with the negative order Sobolev space $W^{-1,2}(0,\pi)$  replaced by a weighted $L^p$ space on $[0,\pi]$.  
By analogy with the deterministic case, we can expect that these solutions are continuous in time with values  in a space  of functions which can be extended holomorphically to some open subset $D$ of $\mathbb{C}$ with  $(0,\pi)\subset D$.  

 For the system \eqref{PROBHEAT1_stoch} our main result in this direction is:

\begin{theorem}\label{th_main}
  For every $\delta\in (0,1)$, $\Theta\in \left(0,\frac{\pi}{4}\right)$ and $\psi_0\in A_\delta^2(D_\Theta)$ the mild solution solution $\psi$
  of \eqref{PROBHEAT1_stoch} takes values in the space $A_\delta^2(D_\Theta)$ defined in \eqref{fara_rho}. Moreover, $\psi$ has a version which is continuous on $[0,\infty)$ with
  values in $A_\delta^2(D_\Theta)$. Finally, the above conclusion is sharp, at least relatively to the scale of Hilbert spaces 
  $$
  \left(A_\delta^2(D_\Theta)\right)_{\di{\delta\in (0,1)}
{\Theta\in (0,\pi/4)}}.
$$ 
More precisely, for \ $\Theta\in \left(0,\frac{\pi}{4}\right)$ \ the solution $\psi$
does not generally take values in $A^2(D_\Theta)$  and for $\delta\in (0,1)$ it does not take values in $A_\delta^2(D_\frac{\pi}{4})$. 
\end{theorem}

We also consider the analogue of \eqref{PROBHEAT1_stoch} with noise in the Neumann boundary data, i.e., the system

\begin{equation}\label{PROBHEAT1_stoch_neu}
   \left\{ \begin{array}{lr}\dfrac{\partial \psi}{\partial t}(t,x)
   = \dfrac{\partial^2 \psi}{\partial x^2}(t,x)& \qquad(t\geqslant 0,\
   x\in (0,\pi)),\\ \ &\ \\ \frac{\partial \psi}{\partial x}(t,0)=  {\dot W}^0_t,
   \ \ \frac{\partial \psi}{\partial x}\left(t,\pi\right)=  {\dot W}_t^\pi & t\in[ 0,\infty),\\ \ &\ \\ \psi(0,x)
   = \psi_0(x)\ \ \ \ \ \ & x\in \left(0,\pi\right), \end{array}\right.
\end{equation}
where $\{W^i_t ,\ t \geqslant 0\}$, with $i \in \{ 1, 2\}$, are independent standard real Wiener processes.

We recall that a concept of mild solution with values in $L^2[0,\pi]$ of  \eqref{PROBHEAT1_stoch_neu} has been introduced in  \cite{DAP_Z_Main}, where the existence, uniqueness and continuity with respect to time of such solutions has been proved. Again, we can expect that these solutions are continuous in time with values  in space formed of much more regular functions. Our main result on the system \eqref{PROBHEAT1_stoch_neu} is:

\begin{theorem}\label{th_main_neu}
  For every $\delta\in (0,1)$, $\Theta\in \left(0,\frac{\pi}{4}\right)$ and $\psi_0$ in the space $A_\delta^{1,2}\left(D_\Theta\right)$, introduced in \eqref{Sob_Berg}, the mild solution solution $\psi$
  of \eqref{PROBHEAT1_stoch_neu} takes values in $A_\delta^{1,2}\left(D_\Theta\right)$. Moreover, $\psi$ has a version continuous in time with values in $A_\delta^{1,2}\left(D_\Theta\right)$. Finally, the above conclusion is sharp, at least relatively to the scale of Hilbert spaces 
  $$
  \left(A_\delta^{1,2}(D_\Theta)\right)_{\di{\delta\in (0,1)}
{\Theta\in (0,\pi/4)}}.
$$ 
More precisely, for \ $\Theta\in \left(0,\frac{\pi}{4}\right)$ \ the solution $\psi$
does not generally take values in $A^{1,2}(D_\Theta)$ 
  and for $\delta\in (0,1)$ it does not take values in $A_\delta^{1,2}(D_\frac{\pi}{4})$.
\end{theorem}

\subsection{Roadmap of the proof of the main results}\label{subsec_roadmap}
In this subsection we briefly explain the strategy of the proofs of \Cref{th_main} and \Cref{th_main_neu}.  
The key point is that the holomorphic regularity of the stochastic trajectories is obtained by combining two ingredients:  
on the one hand, sharp deterministic regularity properties for the corresponding boundary-controlled heat equations; on the other hand, abstract criteria ensuring that stochastic convolutions inherit continuity in time with values in a smaller state space.

We first treat the case of Dirichlet boundary noise.  
The starting point is the deterministic control system associated with the heat equation with Dirichlet boundary inputs, recalled in \Cref{backsys}.  
For this system, the reachable space is known to coincide with the Bergman space $A^2(D_{\frac{\pi}{4}})$.  
Our aim is to show that, after passing to the stochastic problem, one still obtains trajectories in a holomorphic state space, but only in the larger weighted spaces $A_\delta^2(D_\Theta)$ with $\delta\in(0,1)$ and $\Theta\in(0,\frac{\pi}{4})$.

The proof proceeds in two steps.  
First, in \Cref{sec_rhombus} we establish weighted Hilbert--Schmidt estimates for the deterministic input maps associated with the Dirichlet boundary-controlled heat equation.  
These estimates rely on the explicit representation of the input map by complexified heat kernels and on precise summability properties in weighted Bergman spaces on rhombi.  
They yield, in particular, that the input operators introduced in \eqref{def_tilde_phi_dir} are Hilbert--Schmidt from $L^2([0,\tau];\mathbb C^2)$ to $A_\delta^2(D_\Theta)$.

Second, in \Cref{sec_generator} we show that the heat semigroup with Dirichlet boundary conditions restricts to an analytic semigroup on $A_\delta^2(D_\Theta)$.  
This requires proving that the Dirichlet Laplacian admits a natural realization on this weighted Bergman space and that the corresponding resolvent is bounded on suitable sectors.  
Once this semigroup-theoretic step is completed, the abstract results from \Cref{sec2}, and in particular \Cref{prop_reg_stoch}, imply that the stochastic convolution has a version continuous in time with values in $A_\delta^2(D_\Theta)$.  
This proves the first part of \Cref{th_main}.

The sharpness statement in \Cref{th_main} is obtained by going back to the deterministic input maps.  
More precisely, \Cref{prop_sharp} shows that at the critical values $\delta=0$ or $\Theta=\frac{\pi}{4}$ the relevant input map is no longer Hilbert--Schmidt into the corresponding Bergman space.  
By the characterization recalled in \Cref{th_reg_min}, this prevents the stochastic solution from taking values in those critical spaces in general.

The proof of \Cref{th_main_neu} follows the same general pattern, but with the Bergman--Sobolev space $A_\delta^{1,2}(D_\Theta)$ in place of $A_\delta^2(D_\Theta)$.  
The main observation is that the Neumann problem is closely linked to the Dirichlet one through differentiation with respect to the space variable.  
This allows us to transfer a substantial part of the deterministic estimates from the Dirichlet setting to the Neumann setting.  
In particular, the derivative of the Neumann input map can be identified with the Dirichlet input map, which leads to Hilbert--Schmidt estimates in $A_\delta^{1,2}(D_\Theta)$.

The remaining step is to prove that the Neumann heat semigroup restricts to an analytic semigroup on $A_\delta^{1,2}(D_\Theta)$.  
This is carried out in \Cref{sen_meu_main}, together with the density and resolvent estimates needed for the abstract stochastic argument.  
Once these facts are established, \Cref{prop_reg_stoch} again yields continuity in time of the stochastic trajectories with values in the appropriate holomorphic space.  
The optimality of the result is proved by the same mechanism as in the Dirichlet case, namely by combining the failure of the Hilbert--Schmidt property at the critical parameters with \Cref{th_reg_min}.

To summarize, the proof of both main theorems follows the same scheme:
\begin{enumerate}
\item identify a holomorphic state space naturally connected with the reachable space of the deterministic boundary-controlled heat equation;
\item prove that the deterministic input maps satisfy suitable Hilbert--Schmidt estimates in that space;
\item show that the heat semigroup restricts to an analytic semigroup on the same space;
\item apply  abstract stochastic regularity criteria to deduce that the stochastic solution admits a continuous version with values in that holomorphic space;
\item use the failure of the Hilbert--Schmidt property at the critical values to prove sharpness.
\end{enumerate}
This roadmap also explains the organization of the paper: \Cref{sec_rhombus} provides the deterministic Bergman-space estimates for the Dirichlet problem, \Cref{sec_generator} uses them to prove \Cref{th_main}, and \Cref{sen_meu_main} adapts the argument to the Neumann case and proves \Cref{th_main_neu}. The paper is completed by \Cref{backsys} and \Cref{sec2} containing background material and by two appendices collecting the proof of the main resolvent estimate and of several auxiliary properties of Bergman spaces on rhombi.



\section{Some background on well-posed linear control systems}\label{backsys}

 Beginning with its introduction in \cite{DAP_Z_Main}, the formalism used to describe  systems with white noise in the boundary 
conditions has many common points with the theory of well-posed linear boundary control systems introduced in  Weiss
\cite{weiss1989admissibility_con}. This similitude is natural since in both cases the input (the white noise or the deterministic 
control) acts through the boundary. Consequently, the semigroup  formulations of these two types of these problems involve operators with values in a space 
larger then the aimed state space, often called {\em unbounded control operators.} For many deterministic or stochastic PDE systems  involving such operators, we can however
construct solutions with values in the natural state space. This is due to a property of the control operator, which has been  called {\em admissibility} in \cite{weiss1989admissibility_con}
in the deterministic case and {\em stochastic admissibility} in Abreu, Haak and Van Neerven \cite{abreu2013stochastic}, for systems with noise in the boundary conditions.

In this section we gather, for later use, some basic facts about a class of deterministic control systems. These systems are described by a strongly continuous operator semigroup,
simply designed by operator semigroup or just semigroup in the remaining part of this work, and by an admissible control operator.

Most of these results are known, so they are given without proofs. We refer to \cite{weiss1989admissibility_con} or Tucsnak and Weiss \cite[Ch.3]{TucsnakWeiss} for more details and for the proofs.
Moreover, this section also contains two results which seem new:  \Cref{prop_restr} and \Cref{rem_exp_abs} below, for which we provide proofs.

\begin{definition} \label{WPS}
Let $U$ and $X$ be Hilbert spaces. A {\em well-posed linear control
system} with state space $X$ and control space $U$ is a couple $\Sigma=\begin{bmatrix}\mathbb{T} & \Phi\end{bmatrix}$  of families of operators such that
\begin{enumerate}
\item $\mathbb{T}=(\mathbb{T}_t)_{t\geqslant 0}$ is an operator semigroup on $X$;
\item $\Phi=(\Phi_t)_{t\geqslant 0}$ is a family of bounded linear
operators from $L^2([0,\infty);U)$ to $X$ such that
for every $u, v\in L^2([0,\infty);U)$ and all $\tau,t\geqslant 0$ we have
\begin{equation}\label{funceqPhi0}
   \Phi_{\tau+t} (u\ICONT{\tau}v) = \mathbb{T}_t\Phi_\tau u + \Phi_t v \qquad\qquad(t,\tau\geqslant 0),
\end{equation}
where the \emph{$\tau$-concatenation}
of two signals $u$ and $v$, denoted $u\ICONT\tau v$, is the function
\vspace{-2mm}
\begin{equation}\label{DEF_CONC}
u\ICONT\tau v = \begin{cases}
 u(t)&\quad \hbox{for} \ \ t\in[0,\tau),\\
 v(t-\tau)&\quad \hbox{for}\ \ t\geqslant\tau .
 \end{cases}
 \end{equation}
\end{enumerate}
\end{definition}

Let $A:\mathcal{D}(A)\to X$ be the generator of $\mathbb{T}$. 
Let $\beta$ belong to the resolvent set of $A$ and let $X_{-1}$ be the completion of $X$ with respect to the the norm
\begin{equation}\label{f_in_X}
    \|f\|_{X_{-1}}=\|(\beta \mathbb{I}-A)^{-1} f\|_X \qquad\qquad(f\in X).
\end{equation}
 
\begin{remark}\label{rem-1}
The space $X_{-1}$ defined above does not depend on the choice of $\beta$ in the resolvent set of $A$, since $\|\cdot\|_{X_ {-1}}$
is equivalent to the  norm in the dual of $\mathcal{D}(A^*)$ with respect to the pivot space $X$.
\end{remark}

It is known (see, for instance, \cite[Remark 2.10.5]{TucsnakWeiss}) that 
$X\subset X_{-1}$ with continuous and dense embedding and that the original semigroup $\mathbb{T}$ has an extension to $X_{-1}$ that is the image of $\mathbb{T}$ through the unitary operator $\beta I-A\in\mathcal{L}(X,X_{-1})$, where
$\beta$ is in the resolvent set of $A$. We refer to \cite[Remark 2.10.5]{TucsnakWeiss} for a proof of the last statement. This restriction (or extension) will be still denoted by $\mathbb{T}$.
Moreover, the generator of this semigroup, which is a restriction or extension of the original generator, will still be denoted by $A$.

We also recall below an important result which is a particular case of Theorem 3.9 in \cite{weiss1989admissibility_con}.

\begin{theorem}\label{th_rep}
Let $X$ and $U$ be two Hilbert spaces and $\Sigma=\begin{bmatrix}\mathbb{T} & \Phi \end{bmatrix}$ be a well-posed linear control system with state space $X$
and control space $U$. Then there exists a unique operator $B\in \mathcal{L}(U,X_{-1})$ such that
\begin{equation}\label{Phi_t}
   \Phi_t u = \int_0^t\mathbb{T}_{t-\sigma}Bu(\sigma)\, {\rm d}\sigma \qquad\qquad(t>0).
\end{equation}
Moreover, for any $z_0\in X$ and $u\in L^2([0,\infty);U)$ the function
\begin{equation}\label{mild}
   z(t) = \mathbb{T}_t z_0 + \Phi_t u \qquad\qquad(t\geqslant 0),
\end{equation}
is the unique solution  in $C([0,\infty);X)$ (in the sense of \cite[Definition 3.5]{weiss1989admissible}) of
\begin{equation}\label{xdot}
   \dot{z}(t) = A z(t) + B u(t) \,,
\end{equation}
with $z(0)=z_0$. 
\end{theorem}

\begin{remark}\label{lem_construction}
 The existence of the operator $B$ in the proof of Theorem 3.9 in \cite{weiss1989admissibility_con} is obtained constructively. More precisely, we have
 \begin{equation}\label{formula_limita}
B {\rm v}=\lim_{\tau\to 0+} \frac{\Phi_\tau(u({\rm v}))}{\tau} \qquad{\rm in}\qquad X_{-1} \qquad({\rm v}\in U),
 \end{equation}
 where, for every ${\rm v}\in U$, we have that  $u({\rm v})\in L^2[0,\infty);U)$ is  any signal with  
 \[
 [u({\rm v})](t)={\rm v} \qquad\qquad(t\in [0,1]).
 \]
 We refer to the proof of Theorem 3.9 in \cite{weiss1989admissible} for the existence of the limit appearing in \eqref{formula_limita}.
\end{remark}

\begin{remark}\label{rem_admissible}
 In \eqref{Phi_t} the notation $\mathbb{T}_{t-\sigma}$ should be understood as standing for the extension of the original $\mathbb{T}_{t-\sigma}$ to $X_{-1}$, so that $\mathbb{T}_{t-\sigma}Bu(\sigma)\in X_{-1}$ for almost every $\sigma\in [0,t]$. However, in the context of well posed linear control systems, the integral on $[0,t]$ of this expression lies in $X$. Given a semigroup $\mathbb{T}$, an 
 operator $B\in \mathcal{L}(U,X_{-1})$ having the property that for every $u\in L^2([0,\infty);U)$ and every $t\geqslant 0$ the right hand side of \eqref{Phi_t}
 defines an element of $X$ is called an {\em admissible control operator for $\mathbb{T}$.} Thus, thanks to \Cref{th_rep}, a well posed linear control system
 can be alternatively defined by a pair $(A,B)$, with $A$ the generator of an operator semigroup $\mathbb{T}$ on $X$ and  $B\in \mathcal{L}(U,X_{-1})$ an admissible 
 control operator for $\mathbb{T}$.
\end{remark}

We also recall two other concepts which are fundamental in systems theory: the reachable space and null controllability. Firstly, given $\tau>0$, the reachable space in time $\tau$
of a well-posed linear control system $\Sigma=\begin{bmatrix} \mathbb{T} & \Phi \end{bmatrix}$, denoted ${\rm Ran}\, \Phi_\tau$, is the range of the operator $\Phi_\tau$.
We also note that ${\rm Ran}\, \Phi_\tau$, endowed with the norm
\begin{equation}\label{norm_induced}
\|\eta\|_{{\rm Ran}\, \Phi_\tau}=\inf_{\di{u\in L^2([0,\tau];U)}
{\Phi_\tau u=\eta}} \quad \|u\|_{L^2([0,\tau];U)} \qquad\qquad(\eta\in {\rm Ran}\, \Phi_\tau) ,
\end{equation}
is a Hilbert space.

\begin{definition} \label{excont} 
Let $\tau>0$. The well-posed control  system $\begin{bmatrix}\mathbb{T} &\Phi \end{bmatrix}$ is said null-controllable in time in time $\tau$
      if \ ${\rm Ran}\, \Phi_\tau\supset{\rm Ran}\, \mathbb{T}_\tau$.
\end{definition}

The classical result below, due to Fattorini \cite{fatt78} and  Seidman \cite{Seid79} (see also \cite[Proposition 3.1]{MR4474841} for a short proof),
gives an important property of systems which are null controllable in any time.

\begin{proposition}\label{FattSeid}
Assume that the well-posed  linear control system $\Sigma=\begin{bmatrix}\mathbb{T} &\Phi\end{bmatrix}$ is null controllable in any positive time.
Then ${\rm Ran}\, \Phi_\tau$ does not depend on $\tau>0$.
\end{proposition}

The result above justifies the following definition:

\begin{definition}\label{def_reachable_unic}
Let $\Sigma=\begin{bmatrix}\mathbb{T} &\Phi\end{bmatrix}$ be a well-posed linear control system which is null controllable in any positive time. Its reachable space $R_\Sigma$ is
defined as ${\rm Ran}\, \Phi_\tau$ for some (and hence all) $\tau>0$.
\end{definition}

We give below two examples illustrating the concepts and results introduced in this section. These examples  will be repeatedly used in this work.

\begin{example}\label{ex_heat_det}
 It is well known, see, for instance, \cite[Section 10.7]{TucsnakWeiss}, that the  system \eqref{PROBHEAT1} can be written in the form  \eqref{xdot} with
 $X=X_D:=W^{-1,2}(0,\pi)$ and 
 $A=A_D$, where the operator $A_D:\mathcal{D}(A_D)\to X_D$ is defined by \(\mathcal{D}(A_D)=W_0^{1,2}(0,\pi)\), with
 \begin{equation*}
     A_D \varphi=\frac{{\rm d}^2 \varphi }{{\rm d}x^2} \qquad\qquad(\varphi\in \mathcal{D}(A_D)).
 \end{equation*}
 It is well known that $A_D$ is a negative operator on $X_D$, so it generates an analytic semigroup $\mathbb{T}^D$
 on $X_D$.
 In this case the input space is $U=\mathbb{C}^2$ and the admissible control operator $B_D$ is defined by 
 \begin{equation}\label{def_B_primul}
 B_D \begin{bmatrix}  u_0\\ u_\pi \end{bmatrix}  =   \begin{bmatrix}  -u_0 \frac{{\rm d}\delta_0}{{\rm d}x}\\ u_\pi \frac{{\rm d}\delta_\pi}{{\rm d}x} \end{bmatrix} \qquad\qquad\left(\begin{bmatrix}  u_0\\ u_\pi \end{bmatrix}\in \mathbb{C}^2\right),
 \end{equation}
 where $\delta_0$ and $\delta_\pi$ are the Dirac masses concentrated at $x=0$ and $x=\pi$, respectively.
 In other terms, equations \eqref{PROBHEAT1} determine a well-posed linear control system $\Sigma_D$ with state space $X_D:=W^{-1,2}(0,\pi)$ and control space $U=\mathbb{C}^2$.
 It goes back to the classical work of Fattorini and Russell \cite{FatRus2} that this system is null controllable in any positive time so its reachable space
 $R_{\Sigma_D}$ is meaningful according to \Cref{def_reachable_unic}.  This range has been fully determined  in a series of papers, going from Hartmann, Kellay and Tucsnak \cite{HKT_2020}
to Hartmann and Orsoni \cite{hartmann2021separation}, where it has been shown that for every $\tau>0$ we have $R_{\Sigma_D}=A^2\left(D_\frac{\pi}{4}\right)$, where $D_\frac{\pi}{4}$ is defined by \eqref{fara_rho} with $\Theta=\frac{\pi}{4}$.  As mentioned in \Cref{rem_ident}, this means that the states which are reachable at time $\tau>0$ are exactly the smooth functions defined on $(0,\pi)$ which admit an extension to  $D_\frac{\pi}{4}$ which is both holomorphic and square integrable (with respect to the area measure) on $D_\frac{\pi}{4}$.
\end{example}

\begin{example}\label{ex_heat_det_neu}
Consider the system
\begin{equation}\label{PROBHEAT1_neu}
   \left\{ \begin{array}{lr}\dfrac{\partial z}{\partial t}(t,x)
   = \dfrac{\partial^2 z}{\partial x^2}(t,x)& t\geqslant 0,\
   x\in (0,\pi),\\ \ &\ \\ \frac{\partial z}{\partial x}(t,0)=  u_0(t),
   \ \ \frac{\partial z}{\partial x}\left(t,\pi\right)=  u_\pi(t) & t\in[ 0,\infty),\\ \ &\ \\ 
   z(0,x)
= 0\ \ \ \ \ \ & x\in \left(0,\pi\right), \end{array}\right.
\end{equation}
described by the one dimensional heat equation with Neumann boundary control. As in the case of Dirichlet boundary inputs, this system can be written (see, for instance, \cite[Section 10.2]{TucsnakWeiss}) in the form  \eqref{xdot} with
 $X=X_N:=L^2[0,\pi]$ and 
 $A=A_N$, where the operator $A_N:\mathcal{D}(A_N)\to X_N$ is defined by
 \[
 \mathcal{D}(A_N)=\left\{f\in W^{2,2}(0,\pi)\ \ \ |\ \ \frac{{\rm d}f}{{\rm d}x}(0)=\frac{{\rm d}f}{{\rm d}x}(\pi)=0\right\},
 \]
 \begin{equation*}
     A_N \varphi=\frac{{\rm d}^2 \varphi }{{\rm d}x^2} \qquad\qquad(\varphi\in \mathcal{D}(A_N)).
 \end{equation*}
 It is well known that $A_N$ is a negative operator on $X_N$, so it generates an analytic semigroup $\mathbb{T}^N$
 on $X_N$.
 In this case the input space is $U=\mathbb{C}^2$ and the admissible control operator $B_N$ is defined by 
 \begin{equation}\label{def_B_primul_neumann}
 B_N \begin{bmatrix}  u_0\\ u_\pi \end{bmatrix}  = 
   \begin{bmatrix}  u_0 \delta_0\\ -u_\pi \delta_\pi \end{bmatrix} \qquad\qquad\left(\begin{bmatrix}  u_0\\ u_\pi \end{bmatrix}\in \mathbb{C}^2\right),
 \end{equation}
 where $\delta_0$ and $\delta_\pi$ are the Dirac masses concentrated at $x=0$ and $x=\pi$, respectively.

Alternatively, we can say that \eqref{PROBHEAT1_neu} determine a well-posed linear control system with state space $X_N$ and control space $U=\mathbb{C}^2$.
According to \cite{FatRus2}, this system is null controllable in every positive time. Thus  
\Cref{FattSeid} implies that ${\rm Ran}\, \Phi_\tau^N$ does not depend on $\tau>0$. According to \cite{hartmann2021separation}, for every $\tau>0$ we have ${\rm Ran}\, \Phi_\tau^N=A^{1,2}\left(D_\frac{\pi}{4}\right)$ (see \eqref{Sob_Berg} for the definition of this space).
 Again, as mentioned in \Cref{rem_ident}, this means that the states which are reachable at time $\tau>0$ are exactly the smooth functions defined on $(0,\pi)$ which admit an extension to  $D_\frac{\pi}{4}$ having the properties of being holomorphic on $D_\frac{\pi}{4}$ and of being, together with its complex derivative, square integrable on $D_\frac{\pi}{4}$.
\end{example}
 
 The proposition  below gives sufficient conditions guaranteeing that the restriction of a well-posed linear control system to a smaller space is still a well-posed linear control system, with the same control operator. 

\begin{proposition}\label{prop_restr}
    Let $\Sigma =\begin{bmatrix} \mathbb{T} &\Phi \end{bmatrix}$ be a well posed 
    control system with state space $X$ and control space $U$. Let $A$ be the generator of $\mathbb{T}$ and let $B$ be
    the associated control operator defined in \eqref{formula_limita}. Assume 
    that the system is null controllable in any positive time and let the Hilbert space 
    $\tilde X$ be invariant for $\mathbb{T}$, with
    \begin{equation}\label{inclusions}
    R_\Sigma\subset \tilde X \subset X,
    \end{equation}
    with continuous  inclusions, where  $R_\Sigma$ is the reachable space of $\Sigma$ (in the sense of \Cref{def_reachable_unic}). Moreover, assume that $\tilde{\mathbb{T}}:=\mathbb{T}|_{\tilde X}$\
    is an operator semigroup on $\tilde{X}$. Then we have: 
    \begin{enumerate}
   \item  $\tilde \Sigma=\begin{bmatrix}
        \tilde{\mathbb{T}} & \Phi
    \end{bmatrix}$ is a well posed  control system with state space $\tilde X$ and control space $U$. Moreover, 
    the generator $\tilde A$ of  $\tilde{\mathbb{T}}$ is the part of $A$ in $\tilde X$, i.e.
    \begin{equation}\label{gen_tilde}
    \mathcal{D}(\tilde A)=\{f\in \mathcal{D}(A)\ \ |\ \ Af\in \tilde X\},
    \end{equation}
    \begin{equation}\label{op_tilde}
    \tilde A f=A f \qquad\qquad(f\in \mathcal{D}(\tilde A)).
    \end{equation}
    \item 
    Let $\tilde X_{-1}$ be the completion of $\tilde X$ with respect to the norm $f\mapsto \|(\beta I-\tilde A)^{-1}f\|_{\tilde{X}}$, where $\beta$ is in resolvent set of $\tilde A$,
    and recall that $X_{-1}$ is the completion of $X$ with respect to the norm defined in \eqref{f_in_X}. Then, up to a linear isomorphism, we have that $\tilde X_{-1}\subset X_{-1}$
    with continuous inclusion.
    
    \item  Let $\tilde B$ be the control operator associated to $\tilde\Sigma$ via \eqref{formula_limita}, where $X_{-1}$ is replaced by
    $\tilde X_{-1}$.
    Then $\tilde B=B$.
    \end{enumerate}
\end{proposition} 

\begin{proof}
From the facts that  $\mathbb{T}$, $\Phi$ satisfy \eqref{funceqPhi0},   $\tilde {\mathbb{T}}$ is the restriction of $ \mathbb{T}$ to  $ \tilde{X}$ and  \eqref{inclusions} holds, it follows that 
for every $u, v\in L^2([0,\infty);U)$  we have
\begin{equation}\label{funceqPhi01}
   \Phi_{\tau+t} (u\ICONT{\tau}v) = \tilde{\mathbb{T}}_t\Phi_\tau u + \Phi_t v \qquad\qquad(t,\tau\geqslant 0).
\end{equation}
This shows that indeed $\tilde \Sigma$ is a well posed  control system with state space $\tilde X$ and control space $U$.
The fact that the generator $\tilde A$ of $\tilde{\mathbb{T}}$ is given by \eqref{gen_tilde} and \eqref{op_tilde} follows by applying a standard 
semigroup theoretic result, see, for instance, Engel and Nagel \cite[Chapter II, Subsection 2.3]{EN06} or \cite[Proposition 2.4.4]{TucsnakWeiss}. We have thus proved the properties claimed at
the first point in the proposition. 

To tackle the second point, let $\beta$ be a large enough real number, such that $\beta$ is both in the resolvent set of $A$ and in the resolvent set of $\tilde A$. Since the embedding $\tilde X\subset X$ is continuous, it follows that there exists $M>0$ such that for every  $f\in \tilde X$ we have
\[
\|(\beta\mathbb{I}-\tilde A)^{-1} f\|_{\tilde X}\geqslant M\|(\beta\mathbb{I}-\tilde A)^{-1}f\|_X=M\|(\beta\mathbb{I}-A)^{-1}f\|_X.
\]
Hence, for every $f\in \tilde X$ we have
\[
\|f\|_{X_{-1}}
 \geqslant  M\|f\|_{\tilde X_{-1}}.
\]
Therefore the embedding $i$ of $\tilde X$ into $X$ is a continuous operator from $(\tilde X,\|\cdot\|_{\tilde X_{-1}})$ into
$(X,\|\cdot\|_{X_{-1}})$, so it extends uniquely to a continuous linear map
\[
J:\tilde X_{-1}\to X_{-1}.
\]
Moreover, $J$ is injective. Indeed, let $z\in \tilde X_{-1}$ and assume
that $Jz=0$. The operators $(\beta I-A)^{-1}$ and
$(\beta I-\tilde A)^{-1}$ extend by continuity to isometric isomorphisms from
$X_{-1}$ onto $X$ and from $\tilde X_{-1}$ onto $\tilde X$, respectively.
Applying $(\beta I-A)^{-1}$ to the identity $Jz=0$, we obtain
\[
0 = (\beta I-A)^{-1}Jz = i(\beta I-\tilde A)^{-1}z.
\]
Since $i$ is injective, it follows that
\[
(\beta I-\tilde A)^{-1}z=0.
\]
As $(\beta I-\tilde A)^{-1}:\tilde X_{-1}\to \tilde X$ is injective, we conclude
that $z=0$. Thus $J$ is injective so it is a linear isomorphism from $\tilde X_{-1}$ into a subspace of $X_{-1}$.

To prove the third and last claimed conclusion, we recall from \eqref{formula_limita} that
for every ${\rm v}\in U$ the limit $\displaystyle{\lim_{\tau\to 0+} \frac{\Phi_\tau(u({\rm v}))}{\tau}}$ exists in $\tilde X_{-1}$ and in $X_{-1}$
and equals to $\tilde B {\rm v}$ and $B {\rm v}$, respectively. Since $\tilde X_{-1}$ is continuously embedded in $X_{-1}$ it follows that $\tilde B=B$, which ends the proof.
\end{proof}

We end this section with a result concerning the exponential stability of the restriction of an exponentially stable semigroup to an appropriately chosen subspace. 

 \begin{proposition}\label{rem_exp_abs}
Under the assumptions of \Cref{prop_restr}, assume that $\mathbb T$ is exponentially stable on $X$, i.e., there exist $M,\omega>0$ such that
\[
\|\mathbb T_t\|_{\mathcal L(X)}\leqslant M {\rm e}^{-\omega t}
\qquad\qquad (t\geqslant 0).
\]
Then $(\widetilde{\mathbb T}_t)_{t\geqslant 0}$ is exponentially stable on $\widetilde X$. 
\end{proposition}

\begin{proof}
We fix $\tau>0$ and we we remark that from the null controllability of $\Sigma$ and the inclusion $R_\Sigma\subset \tilde X$  it follws that  $\mathbb{T}_\tau$
maps $X$ into $\tilde X$.  Moreover, it is easy to check that $\mathbb T_\tau$, viewed as an operator from $X$ to $\widetilde X$, 
has a closed graph in $X\times \widetilde X$. By the closed graph theorem, it follows that $\mathbb{T}_\tau\in \mathcal{L}(X,\tilde X)$.
Now, let $t\geqslant \tau$ and $x\in \widetilde X$. Using the semigroup property, we obtain
\[
\widetilde{\mathbb T}_t x
=
\mathbb T_\tau \bigl(\mathbb T_{t-\tau}x\bigr).
\]
Hence
\[
\|\widetilde{\mathbb T}_t x\|_{\widetilde X}
\leqslant C_\tau \|\mathbb T_{t-\tau}x\|_X
\leqslant C_\tau M {\rm e}^{-\omega(t-\tau)} \|x\|_X.
\]
Since $x\in \widetilde X$, it follows that
\[
\|\widetilde{\mathbb T}_t x\|_{\widetilde X}
\leqslant C_\tau M c_{\rm e} {\rm e}^{\omega\tau} {\rm e}^{-\omega t}\|x\|_{\widetilde X}
\qquad (t\geqslant \tau),
\]
where $c_{\rm e}$ is the norm of the embedding operator of $\tilde X$ into $X$.
Thus,
\[
\|\widetilde{\mathbb T}_t\|_{\mathcal L(\widetilde X)}
\leqslant C_\tau M c_{\rm e} {\rm e}^{\omega\tau} {\rm e}^{-\omega t}
\qquad (t\geqslant \tau),
\]
so that indeed $\widetilde{\mathbb T}$ is exponentially stable on $\widetilde X$.
\end{proof}

\section{Some background on linear stochastic differential equations in infinite dimensional Hilbert spaces}\label{sec2}

In this section we recall, following Da Prato and Zabczyk \cites{DAP_Z_Book,DAP_Z_Main} (see also Hairer \cite{hairer2023} or Fkirine, Hadd and Rhandi \cite{fkirine2024evolution}) some basic facts
on linear stochastic differential equations in Hilbert spaces. The particularity of the considered class of differential equations is that the white noise acts through an operator with values in a space which is larger then the aimed state space. As already mentioned in \Cref{backsys}, such operators are called unbounded control operators in the control theoretic literature and their use became common knowledge for systems driven by PDEs with inputs acting on the boundary.  

Within this section we continue to use the concepts and the notation introduced in \Cref{backsys}, which means that we consider the Hilbert spaces $U$ (the control space) and $X$ (the state space), which will constantly be identified with their duals. Moreover,
 $A:\mathcal{D}(A)\to X$ is the generator of an operator semigroup
$\mathbb{T}=\left(\mathbb{T}_t\right)_{t\geqslant 0}$ on $X$ and  $B\in \mathcal{L}(U_,X_{-1})$, where the space $X_{-1}$ has been introduced just above \Cref{rem-1}. We recall from \Cref{backsys}, Remark \ref{rem_admissible}, that $B$ is an admissible control operator for $\mathbb{T}$ if and only if
\(\Sigma=\begin{bmatrix} \mathbb{T} & \Phi\end{bmatrix}\), with the family $\Phi=(\Phi_t)_{t\geqslant 0}$ defined in \eqref{Phi_t}, is a well-posed linear control system with state space $X$ and control space $U$, in the sense of \Cref{WPS}. 

With the above notation, we assume that the Hilbert space $U$ and $X$ are separable and we  consider the  stochastic Cauchy problem in $X_{-1}$
\begin{equation}\label{ec_stoh_abs}
{\rm d} Z(t) = A Z(t){\rm d}t + B {\rm d} W(t) \qquad\qquad(t\geqslant 0),    
\end{equation}
\begin{equation}\label{init_stoh_abs}
    Z(0)=Z_0,
\end{equation}
where $W$ is a cylindrical Wiener process on $U$. It is well known, see, for instance, \cite[Theorem 5.4]{DAP_Z_Book} , that, provided that $Z_0\in X_{-1}$, equations \eqref{ec_stoh_abs}, \eqref{init_stoh_abs} have a unique mild solution
in $X_{-1}$ defined by
\begin{equation}\label{duh_stoh}
 Z(t)=  \mathbb{T}_t Z_0 + \int_0^t \mathbb{T}_{t-\sigma} B\, {\rm d} W (\sigma) \qquad\qquad(t\geqslant 0),
\end{equation}
if and only if for some $\tau>0$ we have
\(
\int_0^\tau \|\mathbb{T}_t B\|_{HS(U,X_{-1})}^2 \, {\rm d}t<\infty,
\)
where $\|L\|_{HS(U,X_{-1})}$ stands for the Hilbert-Schmidt norm of an operator $L\in \mathcal{L}(U,X_{-1})$.

In this work we are interested in the situation when the mild solution $Z$ defined by \eqref{duh_stoh} takes values in the space $X$, generally strictly smaller than $X_{-1}$. This property,
called {\em stochastic admissibility} in \cite{abreu2013stochastic} and \cite{fkirine2024evolution}, 
has been characterized in \cite{DAP_Z_Main}.  
We give below this characterization in a form borrowed from \cite{fkirine2024evolution}.

\begin{theorem}\label{th_reg_min}
Assume that the Hilbert spaces $X$ and $U$ are separable, let the operator $A:\mathcal{D}(A)\to X$ be the generator of an operator semigroup on $X$ and let $B\in \mathcal{L}(U,X_{-1})$. Assume that $B$ is an admissible control operator for $\mathbb{T}$. Then, provided that $Z_0\in X$, the mild solution $Z$ of the stochastic equation \eqref{ec_stoh_abs} defined by \eqref{duh_stoh} takes values in $X$
if and only if for some (hence all) $t>0$ the operator $\Phi_t$ defined in \eqref{Phi_t} is Hilbert-Schmidt from $L^2([0,t];U)$ to $X$.
\end{theorem}

A question of interest is whether the mild solution $Z$ of  \eqref{ec_stoh_abs} has a representative which is continuous (in time) with values in $X$.
A sufficient condition to have this property is given in \Cref{prop_corectata} below. To our best knowledge, this result, although implicitely contained in \cite{DAP_Z_Main},  is not explicitly stated in the existing literature. Therefore, for the convenience of the reader and with no claim of originality, we provide a short proof.

\begin{proposition}\label{prop_corectata}
Under the assumptions of \Cref{th_reg_min}, let $0 < \gamma < \frac12$ and let $(\Phi_\tau)_{\tau\geqslant 0}$ be the input maps of the system $(A,B)$. Let $(L_\tau)_{\tau> 0}$ 
be the linear operators defined on $L^\infty([0,\tau],U)$ by
\begin{equation}\label{def_tilde_phi}
L_\tau u = \Phi_\tau (u_\gamma) \qquad\qquad(\tau>0),    
\end{equation}
where 
\[
u_\gamma(\sigma)=(\tau-\sigma)^{-\gamma} u(\sigma) \qquad\qquad(\sigma\in (0,\tau)).
\]
Assume that for some $\tau>0$ we have that $L_\tau $ extends to a Hilbert-Schmidt operator, still denoted by $L_\tau $,
from $L^2([0,\tau];U)$ to $X$. Then the mild solution $Z$ of  \eqref{ec_stoh_abs} has a representative which is continuous (in time) with values in $X$. 
\end{proposition}

\begin{proof}
Under our assumptions, the adjoint $L_\tau^*$ of $L_\tau$ is a Hilbert-Schmidt operator from $X$ to $L^2([0,\tau];U)$. 
To compute $L_\tau^*$ we introduce the Yosida extension of $B^*$, denoted $B^*_\Lambda$, defined by
\[
\mathcal{D}(B_\Lambda^*) =\{x \in X \ \ |\ \lim_{\lambda\to \infty} \lambda B_\Lambda^* (\lambda\mathbb{I}- A)^{-1}x \ \ \hbox{\rm exists\ in}\ \ U\},
\]
\[
B_\Lambda^* x =
\lim_{\lambda\to \infty} \lambda B^* (\lambda \mathbb{I}- A)^{-1}x \qquad\qquad(x\in \mathcal{D}(B^*_\Lambda)).
\]
Since $B^*$ is an admissible observation operator for $\mathbb{T}^*$, it follows from the representation
theorem of Weiss \cite[Theorem 4.5]{weiss1989admissible} that ${\rm Ran}\, \mathbb{T}_t^* \subset \mathcal{D}(B_\Lambda^*)$. 
Combining this fact with the methodology employed in the proof of Theorem 4.4.3 from \cite{TucsnakWeiss}, it is not difficult to check that this adjoint is 
given by     
\begin{equation}\label{adj_pondere}
(L_\tau^*\psi)(\sigma)=(\tau-\sigma)^{-\gamma} B_\Lambda^* \mathbb{T}_{\tau-\sigma}^*\psi \qquad\qquad(\sigma\in (0,\tau),\ \psi\in X).
\end{equation}
Let $(e_n)_{n\in \mathbb{N}}$ be an orthonormal basis in $X$. Since $L_\tau^*$  is  Hilbert-Schmidt from $X$ to $L^2([0,\tau];U)$
it follows that \
\(
\sum_{n\in \mathbb{N}} \left\| L_\tau^* e_n\right\|^2_{L^2([0,\tau],U)} <\infty.
\) 
Combining the above estimate and \eqref{adj_pondere} int follows that 
\[
\int_0^\tau \sum_{n\in \mathbb{N}} (\tau-\sigma)^{-2\gamma} \left\| B_\Lambda^* \mathbb{T}_{\tau-\sigma}^*e_n \right\|_U^2\, {\rm d}\sigma <\infty.
\]
The above inequality can be equivalently written
\begin{equation}\label{good_form}
\int_0^\tau  (\tau-\sigma)^{-2\gamma} \left\| B_\Lambda^* \mathbb{T}_{\tau-\sigma}^* \right\|_{{\rm HS}(X,U)}^2\, {\rm d}\sigma <\infty,
\end{equation}
which implies the conclusion by applying Theorems 2.3 and 2.1 from \cite{DAP_Z_Main}. 
\end{proof}

We next provide two examples which are the counterparts of \Cref{ex_heat_det} and \Cref{ex_heat_det_neu}
when the boundary data are given by white noises.

\begin{example}\label{ex_heat_stoh}
The system \eqref{PROBHEAT1_stoch} can be written in the form \eqref{ec_stoh_abs}, \eqref{init_stoh_abs}, with the choice of the spaces $X=X_D$, $U=\mathbb{C}^2$ and of the operators 
$A=A_D,\ B=B_D$, where $X_D$, $A_D$ and $B_D$ have been defined in  \Cref{ex_heat_det}. 
Let $\Phi=(\Phi_t^D)_{t\geqslant 0}$ be the family of maps defined by \eqref{Phi_t} with $B=B_D$ and $\mathbb{T}=\mathbb{T}^D$ (the operator semigroup generated by $A_D$ on $X_D$).
 As shown in \cite[Proposition 3.1]{DAP_Z_Main}, with this choice of spaces and operators, given $\psi_0\in X_D$, the initial value problem \eqref{ec_stoh_abs}, \eqref{init_stoh_abs} admits a unique solution
mild solution, defined by 
\begin{equation}\label{duh_stoh_dir}
 \psi(t)=  \mathbb{T}_t^D \psi_0 + \int_0^t \mathbb{T}_{t-\sigma}^D B_D \, {\rm d} W (\sigma) \qquad\qquad(t\geqslant 0),
\end{equation}
and this solution has a continuous representative with values in $X_D$. 
\end{example}

\begin{example}\label{rem_neu_stoch}
Similar results are known for the  heat  equation on $[0,\pi]$ with noise in the Neumann boundary conditions. More precisely, we consider the system \eqref{PROBHEAT1_stoch_neu}, where $\{W^i_t ,\ t \geqslant 0\}$, with $i \in \{ 1, 2\}$, are independent standard real Wiener processes.
It is known, see \cite[Proposition 3.2]{DAP_Z_Main}, that the solutions of \eqref{PROBHEAT1_stoch_neu} have an $X_N=L^2[0,\pi]$-continuous version for each   $\psi_0\in X_N$. Moreover, this solution writes
\begin{equation}\label{duh_stoh_neu}
 \psi(t)=  \mathbb{T}_t^N \psi_0 + \int_0^t \mathbb{T}_{t-\sigma}^N B_N \, {\rm d} W (\sigma) \qquad\qquad(t\geqslant 0),
\end{equation}
where $\mathbb{T}_t^N$ and $B_N$ have been defined in \Cref{ex_heat_det_neu}.
\end{example}

We end this section by a result relating the regularity of the processes defined by \eqref{duh_stoh} in terms
of regularity properties for the pair $(A,B)$.

\begin{proposition}\label{prop_reg_stoch}
With the notation and under the assumptions in \Cref{prop_restr}, suppose that   for some (hence all) $t>0$ the operator 
$\Phi_t$ defined in \eqref{Phi_t} is Hilbert-Schmidt from $L^2([0,t];U)$ to $X$ and from $L^2([0,t];U)$ to $\tilde X$.
Then, for every $Z_0\in \tilde X$, the process $Z$ solving \eqref{ec_stoh_abs}, \eqref{init_stoh_abs} and defined by \eqref{duh_stoh}
takes values in $\tilde X$.
Moreover, if for some $\tau,\ \gamma>0$ the operator  $L_{\tau,\gamma}$ defined by \eqref{def_tilde_phi} is Hilbert-Schmidt
from $L^2([0,\tau];U)$ to $\tilde X$ then, for every $Z_0\in \tilde X$, the process $Z$ admits a representative which is continuous in time with values in $\tilde X$.
\end{proposition}

\begin{proof}
    Under our assumptions we can apply \Cref{th_reg_min} to assert that for every $Z_0\in \tilde X$ the process $\tilde Z$ defined by
    \begin{equation}\label{duh_stoh_tilde}
 \tilde Z(t)=  \tilde{\mathbb{T}}_t Z_0 + \int_0^t \tilde{\mathbb{T}}_{t-\sigma} \tilde B\, {\rm d} W (\sigma) \qquad\qquad(t\geqslant 0),
\end{equation}
takes values in $\tilde X$. Since $\tilde{\mathbb{T}}$ is the restriction of $\mathbb{T}$ to $\tilde X$ and, according to \Cref{prop_restr}
we have that $\tilde B=B$ it follows that $\tilde Z=Z$, where $Z$ defined by \eqref{duh_stoh}. Consequently, $Z$ takes values in $\tilde X$.

Finally, if $L_{\tau,\gamma}$  is Hilbert-Schmidt from $L^2([0,\tau];U)$ to $\tilde X$ then, according to \Cref{prop_corectata}, the process $\tilde Z$
has a representative which is continuous in time with values in $\tilde X$. Since $Z=\tilde Z$, the second conclusion follows. 
\end{proof}

\section{From Bergman spaces to the boundary controlled heat equation}\label{sec_rhombus}

In this section we recall some  known results on the Bergman spaces, namely those appearing in the statements of our main results. We focus on the case, already considered in the introduction,
when the functions in these spaces are defined on the rhombus $D_\Theta$ introduced in  \eqref{rtheta} or in an infinite sector
of the complex plane.  More precisely, for $\Theta\in \left(0,\frac{\pi}{2}\right)$ we consider the weighted Bergman spaces  \(A^2_\delta\left(D_\Theta\right)\)
introduced in \eqref{fara_rho}
and we define the Bergman spaces 
\begin{equation}\label{sec_drept}
A^2 \left(\Delta_\Theta\right)=\left\{f\in {\rm Hol} \left(\Delta_\Theta\right)\ \ |\ \ 
\int_{\Delta_\Theta} |f(s)|^2  \, {\rm d }A(s) <\infty\right\},
\end{equation}
\begin{equation}\label{sec_stang}
A^2 \left(\pi-\Delta_\Theta\right)
=\left\{f\in {\rm Hol} \left(\pi-\Delta_\Theta\right)\ \ |\ \ 
\int_{\pi-\Delta_\Theta} |f(s)|^2  \, {\rm d}A(s) <\infty\right\},
\end{equation}
where
\begin{equation}
\Delta_\Theta=\{s\in \mathbb{C}\ \ |\ \ -\Theta<\arg s <\Theta \}.    
\end{equation}
As in the previous section we identify a function $f$ in one of the above spaces with its trace on $(0,\infty)$, or $(-\infty,\pi)$. More precisely, the meaning of the assertion $f\in A^2(\Delta_\Theta)$
is: $f$ is real analytic on $(0,\infty)$ and admits an extension which is holomorphic and square integrable (with respect to the surface measure) on $\Delta_\Theta$. This extension will be still denoted by $f$.

The relevance of these spaces in the study of the heat equation on a half line, with nonhomogeneous
Dirichlet boundary conditions at its extremity, has been first highlighted in  Aikawa, Hayashi, and Saitoh \cite{Saitoh}, from
which we borrow the complex analytic tools used in this section. Moreover, the main result in this section reveals a new connection between the weighted Bergman spaces  \(A^2_\delta\left(D_\Theta\right)\) and the input maps of the system described by the heat equation with Dirichlet boundary controls. To prepare the proof of this main result we first give the proposition below, which provides information on the integrability of the traces on $(0,\pi)$ of functions in \(A^2_\delta\left(D_\Theta\right)\). 

\begin{proposition}\label{lem_intoarsa}
Let  $\Theta\in \left(0,\frac{\pi}{2}\right)$. Then for every $\delta \in (0,1)$ there exists a constant $C_{\Theta,\delta}>0$ such that
\begin{multline}\label{int_pe_raze}
\int_0^{\pi} t^{1+\delta} (\pi-t)^{1+\delta} |f(t)|^2 \, {\rm d}t
\leqslant C_{\Theta,\delta} \|f\|^2_{A^2_\delta(D_\Theta)}
\qquad\quad\left(f\in A^2_\delta(D_\Theta)\right).
\end{multline}
\end{proposition}

\begin{proof}
    Applying Theorems 1.4 and 1.5 from \cite{hartmann2021separation} it follows that every \(f\in A^2_\delta(D_\Theta)\) 
    satisfies
\begin{equation}\label{desco_2}
s^\frac{\delta}{2} (\pi-s)^\frac{\delta}{2}f(s)=g_0(s)+g_\pi (s) \qquad\qquad(s\in D_\Theta),    
\end{equation}
with $g_0\in A^2(\Delta_\Theta)$ and $g_\pi\in A^2(\pi-\Delta_\Theta)$
and 
\begin{equation}\label{eq:estg0g1f}
\|g_0\|_{A^2(\Delta_\Theta)}^2+\|g_\pi\|_{A^2(\pi-\Delta_\Theta)}^2
\leqslant K \|f\|^2_{A^2_\delta(D_\Theta)},
\end{equation}
where $K>0$ is a universal constant.
Applying next the last formula from \cite[Section 5]{Saitoh} it follows that there exists a constant $\tilde K$, depending only on $\theta$,
such that
\begin{equation*}
    \int_0^\pi x (\pi-x)\left( |g_0(x)|^2+ |g_\pi(x)|^2\right)\, {\rm d}x \leqslant \tilde K\left( \|g_0\|_{A^2(\Delta_\Theta)}^2+\|g_\pi\|_{A^2(\pi-\Delta_\Theta)}^2\right).
\end{equation*}
By combining the above estimate, \eqref{desco_2} and \eqref{eq:estg0g1f} we deduce  \eqref{int_pe_raze}. 
\end{proof}

As a consequence of the above result, we obtain:

\begin{corollary}\label{cor_lem_intors}
  Let  $\Theta\in \left(0,\frac{\pi}{2}\right)$, $\delta\in (0,1)$ and let \(A^2_\delta\left(D_\Theta\right)\)
  be the space introduced in \eqref{fara_rho}. Then  $A^2_\delta\left(D_\Theta\right)$ is contained, with continuous inclusion, in the negative order Sobolev space $ W^{-1,2}(0,\pi)$ . 
\end{corollary}

\begin{proof}
We first note there exists an absolute constant $C>0$ such that for every \(\xi\) in the Sobolev space \(W_0^{1,2}(0,\pi)\) we have
    \begin{equation}\label{nu_imediat}
        |\xi(x)|\leqslant  C \min\{\sqrt x,\sqrt{\pi-x}\} \|\xi\|_{W_0^{1,2}(0,\pi)}
        \qquad\qquad(x\in (0,\pi)).
    \end{equation}
Using next the Cauchy-Schwarz inequality and \eqref{int_pe_raze}  it follows that
for every $f\in A_\delta^2(D_\Theta)$ (again, we identify $f$ with its restriction to $(0,\pi)$) we have
 \begin{multline*}
     \left|\int_0^\pi f(x)\xi(x)\, {\rm d}x\right| \leqslant C \|\xi\|_{W_0^{1,2}(0,\pi)} \left(\int_0^\frac{\pi}{2} |f(x)| x^\frac12 \, {\rm d}x+
     \int_\frac{\pi}{2}^\pi |f(x)| (\pi-x)^\frac12 \, {\rm d}x\right)\\
     =C \|\xi\|_{W_0^{1,2}(0,\pi)}\left(\int_0^\frac{\pi}{2} x^\frac{1+\delta}{2}|f(x)| x^{-\frac\delta 2} \, {\rm d}x+
     \int_\frac{\pi}{2}^\pi (\pi-x)^\frac{1+\delta}{2} |f(x)| (\pi-x)^{-\frac\delta 2} \, {\rm d}x\right)\\
     \leqslant K \|\xi\|_{W_0^{1,2}(0,\pi)},
 \end{multline*}
 where $C$ is a universal constant and $K$ is a constant depending on $\|f\|_{A_\delta^2(D_\Theta)}$. 
\end{proof}

We turn now to the input maps of the system described by the heat equation with Dirichlet boundary controls. More precisely,
consider the maps $(\Phi_\tau^D)_{\tau \geqslant 0}$  defined by
\begin{equation}\label{DEFFITAUSIMPLU}
\Phi_\tau^D \begin{bmatrix} u_0\\ u_\pi\end{bmatrix}=z(\tau,\cdot)
\qquad\qquad(\tau \geqslant 0,\ u_0,\ u_\pi\in L^2[0,\tau]),
\end{equation}
where $z$ is the solution of \eqref{PROBHEAT1}. Alternatively, the input maps are defined by \eqref{Phi_t},
with $X,\ U,\ A$ and $B$ chosen as in \Cref{ex_heat_det}. Within this context we recall from \Cref{ex_heat_det} that $B$ is an admissible control operator for the semigroup generated by $A$.
Moreover, putting together results from Hartmann, Kellay and Tucsnak \cite{HKT_2020} and from  \cite{hartmann2021separation}
we have:

\

\

\

\begin{theorem}\label{rephrased}
Let $(\Phi_\tau^D)_{\tau \geqslant 0}$ be the input maps defined in \eqref{DEFFITAUSIMPLU}. Then for every $\tau>0$ the range of $\Phi_\tau^D$, denoted ${\rm Ran}\, \Phi_\tau^D$, coincides with $A^2\left(D_{\frac{\pi}{4}}\right)$.   
Moreover, 
\begin{multline}\label{DEFOP1}
\left(\Phi_\tau^D \begin{bmatrix} u_0\\ u_\pi\end{bmatrix}\right)(s) =\int_0^\tau  \frac{\partial K_0}{\partial s} (\tau-\sigma,s) u_0 (\sigma) \, {\rm d}\sigma\\
 + \int_0^\tau \frac{\partial K_\pi}{\partial s}(\tau-\sigma,s) u_\pi (\sigma) \, {\rm d}\sigma \qquad(\tau>0,\ u_0,\ u_\pi \in L^2[0,\tau],\ s\in D_\frac{\pi}{4}),
\end{multline}
 and
\begin{equation}\label{MAGIC0}
K_0(\sigma,s)=-\sqrt{\frac{1}{\pi\sigma}} \sum_{m\in \mathbb{Z}} \ {\rm e}^{-\frac{(s+2m \pi)^2}{4\sigma}}\qquad\qquad(\sigma>0,\ s\in D_\frac{\pi}{4}),
\end{equation}
\begin{equation}\label{MAGICPI}
K_\pi(\sigma,s)=K_0 (\sigma,\pi-s) \qquad\qquad(\sigma>0, s\in D_\frac{\pi}{4}),
\end{equation}
where the  series in \eqref{MAGIC0} converges in $A^2\left(D_{\frac{\pi}{4}}\right)$,
\end{theorem}

We are now in a position to state the main result in this section:

\begin{proposition}\label{expint}
Let $\tau>0$, $\delta\in (0,1)$ and $0<\Theta<\frac{\pi}{4}$. Then for every $\gamma \in \left[0,\frac{\delta}{4}\right)$ the input map $\Phi_\tau^D$ defined in \eqref{DEFFITAUSIMPLU} satisfies
\begin{equation*}
\sum_{n\in \mathbb{Z}} \left\|\Phi_\tau^D\left(\begin{bmatrix}{\rm e}_{n,\tau,\gamma}\\ 
0\end{bmatrix}\right)\right\|^2_{A^2_\delta(D_\Theta)}
+  \sum_{n\in \mathbb{Z}} \left\|\Phi_\tau^D\left(\begin{bmatrix} 0 \\ 
{\rm e}_{n,\tau,\gamma}\end{bmatrix}\right)\right\|^2_{A^2_\delta(D_\Theta)} < \infty ,
\end{equation*}
where 
\begin{equation}\label{dexp_timp}
{\rm e}_{n,\tau,\gamma}(t)= (\tau-t)^{-\gamma}{\rm e}^{\frac{ in\tau t}{2\pi}}
\qquad\qquad(n\in \mathbb{Z},\ t\in [0,\tau]).
\end{equation} 
Consequently, the operators $(L_\tau^D)_{\tau>0}$ defined by
\begin{equation}\label{def_tilde_phi_dir}
L_\tau^D u = \Phi_\tau^D (u_\gamma) \qquad\qquad(\tau>0),    
\end{equation}
where 
\[
u_\gamma(\sigma)=(\tau-\sigma)^{-\gamma} u(\sigma) \qquad\qquad(\sigma\in (0,\tau))
\]
are Hilbert-Schmidt  from $L^2\left([0,\tau];\mathbb{C}^2\right)$ to $A^2_\delta(D_\Theta)$.
\end{proposition}

Our proof of the above result requires some notation and two lemmas.
Firstly, for $\Theta\in \left(0,\frac{\pi}{4}\right)$, $\delta\in (0,1)$, $\gamma>0$, $s\in D_\Theta$
and  $m,\ n\in \mathbb{Z}$, we denote
\begin{equation}\label{untilde}
\Gamma_{mn}^{(0)}(s)=\frac{1}{2\sqrt\pi}\int_0^{2\pi} 
\frac{{\rm e}^{-\frac{(s+2m\pi)^2}{4(2\pi-\sigma)}}}{(2\pi-\sigma)^{\frac{3}{2}+\gamma}} (s+2m\pi)\, {\rm e}^{in\sigma}\, {\rm d}\sigma,
\end{equation}
\begin{equation}\label{deuxtilde}
\Gamma_{mn}^{(\pi)}(s)=\frac{1}{2\sqrt\pi}\int_0^{2\pi} 
\frac{{\rm e}^{-\frac{(s+(2m-1)\pi)^2}{4(2\pi-\sigma)}}}{(2\pi-\sigma)^{\frac{3}{2}+\gamma}} \left[(2m+1)\pi-s\right] {\rm e}^{in\sigma}\, {\rm d}\sigma.
\end{equation}

The first of the above mentioned  lemmas states as follows:

\begin{lemma} \label{th_jap_bis}  Let $\Theta\in \left(0,\frac{\pi}{4}\right)$, $\delta\in (0,1)$ and $\gamma\in \left[0,\frac{\delta}{4}\right)$. 
Then, using the notation introduced in \eqref{untilde}, \eqref{deuxtilde}, we have 
\[
\sum_{n\in \mathbb{Z}} \left(\|\Gamma_{0n}^{(0)}\|_{A^2_\delta(D_\Theta)}^2+\|\Gamma_{0n}^{(\pi)}\|_{A^2_\delta(D_\Theta)}^2\right) < \infty.
\]
\end{lemma}

\begin{proof}
     We clearly have that for every $s\in D_\frac{\pi}{4}$ the function
    \[
    \sigma\mapsto \frac{{\rm e}^{-\frac{s^2}{4(2\pi-\sigma)}}}{(2\pi-\sigma)^{\frac{3}{2}+\gamma}} s
    \]
    lies in $L^2[0,2\pi]$. Applying Parseval's theorem it follows that for every $s\in D_\frac{\pi}{4}$ we have
    \begin{equation}\label{numar_sa_fie}
        \sum_{n\in \mathbb{Z}} |\Gamma_{0n}^{(0)}(s)|^2=\int_0^{2\pi } \left|\frac{{\rm e}^{-\frac{s^2}{4(2\pi-\sigma)}}}{(2\pi-\sigma)^{\frac{3}{2}+\gamma}} s\right|^2 \, {\rm d}\sigma = \int_0^{2\pi } \frac{{\rm e}^{-\frac{\Re\,(s^2)}{2(2\pi-\sigma)}}}{(2\pi-\sigma)^{3+2\gamma}} |s|^2 \, {\rm d}\sigma.
    \end{equation}
Using the change of variables $2\pi-\sigma =\xi \Re\, (s^2)$ we obtain that
\begin{multline}\label{noua_linie}
\int_0^{2\pi } \frac{{\rm e}^{-\frac{\Re\,(s^2)}{2(2\pi-\sigma)}}}{(2\pi-\sigma)^{3+2\gamma}}
|s|^2 \, {\rm d}\sigma=
\frac{|s|^2}{\left(\Re\, (s^2)\right)^{2+2\gamma}}\int_0^{\frac{2\pi}{{\Re\, (s^2)}}}
\frac{{\rm e}^{-\frac{1}{2\xi}}}{\xi^{3+2\gamma}} \, {\rm d}\xi  \\
\leqslant  \frac{C|s|^2}{\left(\Re\, (s^2)\right)^{2+2\gamma}}.
\end{multline}

From the above estimate and \eqref{numar_sa_fie} we deduce that
\begin{align*}
\sum_{n\in\mathbb{Z}} \|\Gamma_{0n}^{(0)}\|_{A_\delta^2(D_\Theta)}^2&=\sum_{n\in\mathbb{Z}}\int_{D_\Theta}|s|^\delta |\pi-s|^\delta \left| \Gamma_{0n}^{(0)}(s)\right|^2\,{\rm d}A(s)\\ & \leqslant C \int_{D_\Theta}|s|^{2+\delta} |\pi-s|^\delta \frac{1}{\left(\Re\, (s^2)\right)^{2+2\gamma}}\,{\rm d}A(s)\\
 & = C  \int_{-\Theta}^\Theta \int_0^{r_\theta} \left|\pi-r {\rm e}^{i\theta}\right|^\delta   \frac{r^{-1+\delta -4\gamma}}{\cos^{2+2\gamma}(2 \theta)} \,{\rm d}r\,{\rm d}\theta,
\end{align*}
which, for $0\leqslant \gamma<\frac{\delta}{4}$, implies that
\begin{equation}\label{eq:egmN0}
  \sum_{n\in \mathbb{Z}}  \|\Gamma_{0n}^{(0)}\|_{A_\delta^2(D_\Theta)}^2 <\infty.
\end{equation}
The fact that
\[
\sum_{n\in \mathbb{Z}}\|\Gamma_{0n}^{(\pi)}\|_{A^2_\delta(D_\Theta)}^2<\infty,
\]
can be checked similarly, which ends the proof.
\end{proof}

The last preparatory lemma before proving the main result in this section is: 

\begin{lemma}\label{ceilalti}
Let $\Theta\in \left(0,\frac{\pi}{4}\right]$, $\delta\in [0,1)$ and $\gamma\in \left[0,\frac{\delta}{4}\right]$. 
Then, using the notation introduced in \eqref{untilde}, \eqref{deuxtilde}, there exist two constants $K,\ c>0$ such that for every $m\in \mathbb{Z}^*$   we have
\begin{equation}\label{est_cei}
  \sum_{n\in \mathbb{Z}} \left( |\Gamma_{mn}^{(0)}(s)|^2+|\Gamma_{mn}^{(\pi)}(s)|^2\right)\leqslant K {\rm e}^{-\frac{c m^2}{2\pi}} \qquad\qquad\left(s\in D_\Theta\right).
\end{equation}
\end{lemma}

\begin{proof}
Repeating the arguments from the beginning of \Cref{th_jap_bis} it follows that
for every $s\in D_{\frac{\pi}{4}}$ and $m\in \mathbb{Z}$ we have
    \begin{equation}\label{numar_sa_fie_cei}
        \sum_{n\in \mathbb{Z}} |\Gamma_{mn}^{(0)}(s)|^2 = \int_0^{2\pi }
        \frac{{\rm e}^{-\frac{\Re\,((s+2m\pi)^2)}{2(2\pi-\sigma)}}}{(2\pi-\sigma)^{3+2\gamma}} |s+2m\pi|^2 \, {\rm d}\sigma ,
    \end{equation}
    \begin{equation}\label{numar_sa_fie_cei_tilde}
        \sum_{n\in \mathbb{Z}} |\Gamma_{mn}^{(\pi)}(s)|^2 = \int_0^{2\pi } 
        \frac{{\rm e}^{-\frac{\Re\,((s+(2m-1)\pi)^2)}{2(2\pi-\sigma)}}}{(2\pi-\sigma)^{3+2\gamma}} |s+(2m-1)\pi|^2 \, {\rm d}\sigma .
    \end{equation}
We next note that there exists $a,\ b>0$ such that for every $k\in \mathbb{Z}\setminus\{-1,0\}$ we have
 $$\left|(s+k\pi) {\rm e}^{-\frac{(s+k\pi)^2}{4(2\pi-\sigma)}}\right|^2\leqslant a k^2 {\rm e}^{\frac{-b k^2}{(2\pi-\sigma)}} \qquad\qquad \left(s\in D_{\frac{\pi}{4}}
 \right).$$
It follows that for every $s\in D_{\frac{\pi}{4}}$ and for every $k\in \mathbb{Z}\setminus\{-1,0\}$ we have
 \begin{multline*}
 \int_0^{2\pi}  \frac{\Big|(s+k\pi) 
 {\rm e}^{-\frac{(s+k\pi)^2}{4(2\pi-\sigma)}}\Big|^2}{(2\pi-\sigma)^{3+2\gamma}}\, {\rm d}\sigma
 \leqslant a \int_0^{2\pi} \frac{ k^2 {\rm e}^{\frac{-b k^2}{t}}}{ t^{3+2\gamma}}\, {\rm d}t\\
 =\frac{a{\rm e}^{-\frac{bk^2}{2\pi}}}{b(bk^2)^{1+2\gamma}}
 \int_{\frac{b k^2}{2\pi}}^{\infty} u^{1+2\gamma} {\rm e}^{-u+\frac{bk^2}{2\pi}}\, {\rm d}u \\
 =\frac{a{\rm e}^{-\frac{bk^2}{2\pi}}}{b(2\pi b)^{1+2\gamma}}\int_{0}^\infty \left(\frac{2\pi}{k^2} u+b\right)^{1+2\gamma}{\rm e}^{-u}\,{\rm d}u\\
 \leqslant \frac{a{\rm e}^{-\frac{bk^2}{2\pi}}}{b(2\pi b)^{1+2\gamma}}\int_{0}^\infty \left(2\pi u+b\right)^{1+2\gamma}{\rm e}^{-u}\,{\rm d}u
 =  C {\rm e}^{-\frac{b k^2}{2\pi}}.
 \end{multline*}
Combining the above estimate with \eqref{numar_sa_fie_cei}, \eqref{numar_sa_fie_cei_tilde} yields the announced conclusion.
\end{proof}

We are now in a position to prove the main result in this section.

\begin{proof}[Proof of \Cref{expint}]
We first note that from \Cref{rephrased} it follows that
\begin{equation}\label{fi_suma}\left\{\begin{array}{lcl}
  \displaystyle{\Phi_{2\pi}^D\left(\begin{bmatrix}{\rm e}_{n,\tau,\gamma}\\ 
0\end{bmatrix}\right) (s)}&=&\displaystyle{\sum_{m\in \mathbb{Z}} \Gamma_{mn}^{(0)}(s)} \\  \displaystyle{\Phi_{2\pi}^D\left(\begin{bmatrix}0\\{\rm e}_{n,\tau,\gamma}\end{bmatrix}\right) (s)}&=&\displaystyle{\sum_{m\in \mathbb{Z}}\Gamma_{mn}^{(\pi)}(s)} \end{array}\right.
\qquad\qquad(n\in \mathbb{N},\ s\in D_\Theta),
\end{equation}
where $\Gamma_{mn}^{(0)}$ and $\Gamma_{mn}^{(\pi)}$ have been defined in \eqref{untilde} and \eqref{deuxtilde}, respectively.

On the other hand, from \Cref{th_jap_bis} and \Cref{ceilalti} it follows that
\begin{equation*}
  \sum_{m,n\in \mathbb{Z}} \left( \|\Gamma_{mn}^{(0)}\|_{A_\delta^2(D_\Theta)}^2+\|\Gamma_{mn}^{(\pi)}\|_{A_\delta^2(D_\Theta)}^2\right)<\infty.
\end{equation*}
Combining the above estimate and \eqref{fi_suma} yields the announced conclusion for $\tau=2\pi$. Finally, it is easy to check  that this conclusion holds for every $\tau>0$.
\end{proof}

We end this section with the following result: 

\begin{proposition}\label{prop_sharp}
Let $\tau>0$ and assume that either  $\delta=0$ and $\Theta\in \left(0,\frac{\pi}{4}\right)$
or that $\delta\in (0,1)$ and $\Theta=\frac{\pi}{4}$. Then
the input map $\Phi_\tau^D$ defined in \eqref{DEFFITAUSIMPLU} is not
Hilbert-Schmidt  from $L^2\left([0,\tau];\mathbb{C}^2\right)$ into $A^2_\delta(D_\Theta)$.    
\end{proposition}

\begin{proof}
It clearly suffices to prove the result for $\tau=2\pi$. To achive this goal we first note that
from \eqref{fi_suma} with $\gamma=0$, it follows that for each $n \in \mathbb{Z}$ we have
\begin{equation}\label{fi_suma_part}
  \Phi_{2\pi}^D\left(\begin{bmatrix}{\rm e}_{n,\tau,0}\\ 
0\end{bmatrix}\right) = \Gamma_{0n}^{(0)}+\sum_{m\in \mathbb{Z}^*} \Gamma_{mn}^{(0)},
\end{equation}
where $\left(\Gamma_{mn}^{(0)}\right)_{m,n\in \mathbb{Z}}$  has been defined in \eqref{untilde}.
Moreover, using \eqref{est_cei}
it follows that for every  $s\in D_\frac{\pi}{4}$ we have
\begin{multline*}
\sum_{n\in \mathbb{Z}}\left|\sum_{m\in \mathbb{Z}^*} \Gamma_{mn}^{(0)}(s) \right|^2
\leqslant \frac{\pi^2}{3} \sum_{n\in \mathbb{Z}}\sum_{m\in \mathbb{Z}^*} m^2\left|\Gamma_{mn}^{(0)}(s) \right|^2
=\frac{\pi^2}{3} \sum_{m\in \mathbb{Z}^*} m^2 \sum_{n\in \mathbb{Z}} \left|\Gamma_{mn}^{(0)}(s) \right|^2\\
\leqslant K \sum_{m\in \mathbb{Z}^*} m^2 {\rm e}^{-\frac{c m^2}{2\pi}},
\end{multline*}
so that we have
\begin{equation}\label{aprtine_l2}
\left(\sum_{m\in \mathbb{Z}^*} \Gamma_{mn}^{(0)}\right)_{n\in \mathbb{Z}}\in l^2\left(\mathbb{Z},A^2\left(D_\frac{\pi}{4}\right)\right).
\end{equation}
On the other hand, by combining \eqref{numar_sa_fie} and \eqref{noua_linie} (with $\gamma=0$), it follows that
for every $s\in D_\frac{\pi}{4}$ we have
\[
 \sum_{n\in \mathbb{Z}} |\Gamma_{0n}^{(0)}(s)|^2=\frac{|s|^2}{\left(\Re\, (s^2)\right)^{2}}\int_0^{\frac{2\pi}{{\Re\, (s^2)}}}
\frac{{\rm e}^{-\frac{1}{2\xi}}}{\xi^{3}} \, {\rm d}\xi.
\]
The above formula clearly implies that if $\delta=0$ and $\Theta\in \left(0,\frac{\pi}{4}\right)$
or if $\delta\in (0,1)$ and $\Theta=\frac{\pi}{4}$ then, due to its singularity at $s=0$ (respectively at $\arg s=\frac{\pi}{4}$), the map
\[
s\mapsto s^\delta (\pi-s)^\delta \sum_{n\in \mathbb{Z}} |\Gamma_{0n}^1(s)|^2
\]
is not integrable on $D_\Theta$. Thus
\[
\left(\Gamma_{0n}^{(0)}(s)\right)_{n\in \mathbb{Z}}\not \in l^2\left(\mathbb{Z},A_\delta^2(D_\Theta)\right).
\]
Combining the above fact with \eqref{fi_suma_part} and \eqref{aprtine_l2} it follows that
\[
\left(\Phi_{2\pi}^D\left(\begin{bmatrix}{\rm e}_{n,\tau,0}\\ 
0\end{bmatrix}\right)\right)_{n\in \mathbb{Z}} \not \in l^2\left(\mathbb{Z},A_\delta^2(D_\Theta)\right),
\]
which yields the announced conclusion.
\end{proof}

\section{Proof of \Cref{th_main}}\label{sec_generator}

This section is devoted to the proof of our main result on the case of Dirichlet boundary noise. Our strategy consists in applying \Cref{th_reg_min} and \Cref{prop_reg_stoch} with the appropriate choice of spaces and operators. We continue to use the notation introduced in the previous sections and we introduce some new ones. 

 We denote $\widetilde X_{D,\delta,\Theta}=A^2_\delta(D_\Theta)$, with  $\delta \in \left(0,1\right)$, $\Theta\in \left(0,\frac{\pi}{4}\right)$. The results below are essentially independent of the choice 
 of $\delta$ and $\Theta$ in the above range, therefore we will use the simplified notation $\widetilde X_D$ for $\widetilde X_{D,\delta,\Theta}$. 

\begin{remark}\label{rem_urma_noua}
  If $f\in \widetilde X_D$ with $f''\in \widetilde X_D$, then according to  \Cref{cor_lem_intors}, we have that
 $\frac{{\rm d}^2f}{{\rm d}x^2}\in W^{-1,2}(0,\pi)$, thus $f\in W^{1,2}(0,\pi)\subset C[0,\pi]$. Consequently, $f(0)$ and $f(\pi)$ can be defined as the limits of 
 $f(x,0)$ when $x\in (0,\pi)$ tends to $0$ and to $\pi$, respectively. 
\end{remark}

Consider the operator $\widetilde A_D:\mathcal{D}(\widetilde A_D)\to \widetilde X_D$
defined by
\begin{equation}\label{def_dom_A}
\mathcal{D}(\widetilde A_D)=\left\{f\in \widetilde X_D\ \ | \ \ f''\in \widetilde X_D,\ f(0)=f(\pi)=0\right\},  
\end{equation}
where $f(0)$ and $f(\pi)$ are defined as in \Cref{rem_urma_noua}, and 
\begin{equation}\label{def_op_A}
\widetilde A_D f = f'' \qquad\qquad(f\in \mathcal{D}(\widetilde A_D)).    
\end{equation}

The main new ingredient we bring in this section is:

\begin{theorem}\label{th_hol_tilde}
The operator $\widetilde A_D$ defined in \eqref{def_dom_A}, \eqref{def_op_A} generates a bounded analytic
semigroup  $\widetilde{\mathbb{T}}^D$ on $\widetilde X_D$.  Moreover, $\widetilde{\mathbb{T}}^D$ is the restriction to $\widetilde X_D$ of the analytic semigroup $\mathbb{T}^D$  on $X$ generated 
by the Dirichlet Laplacian (i.e., the operator $A_D$
 introduced in \Cref{ex_heat_det}).
\end{theorem}

\begin{remark}\label{rem_exp_dir}
It is well-known that the semigroup $\mathbb{T}^D$ is exponentially stable in $X_D$.
We can thus apply \Cref{rem_exp_abs} to conclude that $(\widetilde{\mathbb T}^D)$ is exponentially stable on $\widetilde X_D$.
\end{remark}

An important ingredient of the proof of \Cref{th_hol_tilde} is the following approximation result:

\begin{proposition}\label{thm:main}
Let $\delta \in \left(0,1\right)$, $\Theta\in \left(0,\frac{\pi}{4}\right)$ and let $p$ be any polynomial on $\mathbb{C}$. Then there exists a sequence of functions $(g_n)_{n\in \mathbb{N}}$,
holomorphic in $D_\Theta$ and continuous on $\overline{D_\Theta}$,  such that:
\begin{enumerate}[label=\textup{(\alph*)}]
\item $g_n(0)=g_n(\pi)=0$ for all $n\in \mathbb{N}$;
\item $\|s^{\frac{\delta}{2}} (\pi-s)^{\frac{\delta}{2}}  g_n-p\|_{L^2(D_\Theta)}\to 0$ as $n\to\infty$;
\item  for every $n\in \mathbb{N}$ we have that $g_n''(s)
\in A^2(D_\Theta)$.
\end{enumerate}
\end{proposition}

\begin{proof}
    Let $(\phi_n)_{n\in \mathbb{N}}$ and $(\psi_n)_{n\in \mathbb{N}}$ be the sequences of polynomials defined  by
\begin{equation}\label{fi_si_psi}
\phi_n(s) :=  \left(\frac{s}{\pi}\right)^n, \qquad \psi_n(s) :=
 \left(\frac{\pi-s}{\pi}\right)^n \qquad\qquad(n\in \mathbb{N}).
\end{equation}
Then $\phi_n$ and $\psi_n$ satisfy
\begin{equation}\label{prop_phi}
\phi_n(\pi) = 1, \quad \phi_n(0) = 0 \qquad\qquad(n\in \mathbb{N}),
\end{equation}
\begin{equation}\label{prop_psi}
\psi_n(\pi) = 0, \quad \psi_n(0) = 1  \qquad\qquad(n\in \mathbb{N}) .
\end{equation}
We next define
\begin{equation}\label{def_gn}
f_n(s):=(1-\phi_n(s))^2(1-\psi_n(s))^2 \qquad\qquad(n\in \mathbb{N},\ s\in \mathbb{C}),
\end{equation}
and we note that
\begin{enumerate}[label=\textup{(\roman*)}]
\item $(f_n)$ are polynomials;
\item For each $n\in \mathbb{N}$, $f_n$ has a zero of order at least $2$ at $0$ and at $\pi$;
\item $f_n\to 1$ in $L^2(D_\Theta)$.
\end{enumerate}
The first two properties above are a direct consequence of the construction of $f_n$ (namely of \eqref{prop_phi} and \eqref{prop_psi}), whereas  the third one follows by applying 
the facts that
\[
\lim_{n\to \infty} f_n(s)=1 \qquad\qquad(s\in D_\Theta),
\]
\[
|f_n(s)|\leqslant 16 \qquad\qquad(n\in \mathbb{N},\ s\in D_\Theta),
\]
and the dominated convergence theorem.

For $p$ an arbitrary polynomial we define 
\[
g_n(s):= s^{-\frac{\delta}{2}} (\pi-s)^{-\frac{\delta}{2}} p(s)\,f_n(s) \qquad\qquad(n\in \mathbb{N},\ s\in \mathbb{C}),
\]
where $f_n$ has been defined in \eqref{def_gn}. Using properties {\it i)} and {\it ii)} above it follows that 
$(g_n)_{n\in \mathbb{N}}$ are
holomorphic in $D_\Theta$, continuous on $\overline{D_\Theta}$
and satisfy  the conclusion {\it (a)} in the statement of the proposition. 

To prove  {\it (b)}, it suffices to note that
\begin{multline*}
\| s^\frac{\delta}{2}(\pi-s)^\frac{\delta}{2} g_n-p\|_{L^2(D_\Theta)}
=
\left\|p\left(  f_n-1\right)\right\|_{L^2(D_\Theta)}\\
\leqslant
\|p\|_{L^\infty(\overline{D_\Theta})}\,\|f_n-1\|_{L^2(D_\Theta)}
\end{multline*}
and to use property {\it iii)} above.

In order to prove conclusion {\it (c)}  we only have to analyze the behavior of
$g_n''(s)$ near $s=0$ and $s=\pi$.

Since $f_n$ has a zero of order at least $2$ at $0$, there exists a polynomial function
$h_{n,0}$ such that
\begin{equation}\label{eq:pn_factor_0}
g_n(s)=s^{2-\frac{\delta}{2}} (\pi-s)^{-\frac{\delta}{2}} h_{n,0}(s) \qquad\qquad (n\in \mathbb{N},\ s\in\mathbb{C}).
\end{equation}
Choose $\varepsilon\in(0,\pi/2)$ so that $\overline{B(0,\varepsilon)}$ does not intersect $\{\pi\}$.
Then the function $s\mapsto (\pi-s)^{-\frac{\delta}{2}}$  and $h_{n,0}$ are holomorphic and bounded on $\overline{B(0,\varepsilon)}$. Set
\[
H_{n,0}(s):=(\pi-s)^{-\frac{\delta}{2}} h_{n,0}(s),
\]
which is holomorphic on $B(0,\varepsilon)$ and bounded together with its first and second derivatives on $B(0,\varepsilon)$.
Then \eqref{eq:pn_factor_0} yields that
\begin{equation}\label{eq:Fn_as_power_0}
g_n(s)=s^{2-\frac{\delta}{2}}H_{n,0}(s)\qquad (|s|<\varepsilon).
\end{equation}
Differentiating twice we obtain:
\begin{multline}\notag
g_n''(s)
=
\left(2-\frac{\delta}{2}\right) \left(1-\frac{\delta}{2}\right)s^{-\frac{\delta}{2}} H_{n,0}(s)\\
+2\left(2-\frac{\delta}{2}\right)s^{1-\frac{\delta}{2}}H_{n,0}'(s)
+s^{2-\frac{\delta}{2}}H_{n,0}''(s).\
\end{multline}
Since $H_{n,0}$, $H_{n,0}'$, $H_{n,0}''$ are bounded on $B(0,\varepsilon)$, there exists $C_{n,\varepsilon}>0$ such that
\begin{equation}\label{eq:Fnpp_bound_0}
|g_n''(s)|\leqslant C_{n,\varepsilon}\,|s|^{-\frac{\delta}{2}}
\qquad \qquad(0<|s|<\varepsilon).
\end{equation}
The above estimate clearly implies that
\[
\int_{D\cap\{|s|<\varepsilon\}} |g_n''(s)|^2\,\mathrm{d}A(s)
\leqslant
\int_{-\Theta}^\Theta \int_0^\varepsilon C_{n,\varepsilon}^2 r^{-\delta}\, r\,\mathrm{d}r\,\mathrm{d}\varphi
=
2 C_{n,\varepsilon}^2\,\Theta \int_0^\varepsilon r^{1-\delta}\,\mathrm{d}r.
\]
The last integral is finite since $\delta\in (0,1)$ (this even holds for $\delta\in (0,2)$), so that 
\begin{equation*}
\int_{D\cap\{|s|<\varepsilon\}} |g_n''(s)|^2\,\mathrm{d}A(s) <\infty.
\end{equation*}
Similarly we can check that 
\begin{equation*}
\int_{D\cap\{|\pi-s|<\varepsilon\}} |g_n''(s)|^2\,\mathrm{d}A(s) < \infty,
\end{equation*}
which end the proof of assertion {\it iii)}, hence of our proposition. 
\end{proof}

As a consequence of \Cref{thm:main}  we obtain:

\begin{corollary}\label{cor_imp}
 For $\delta\in (0,1)$ we denote by $Y$ the vector space formed of all the functions $g$ which are holomorphic in $D_\Theta$ and
 continuous on $\overline{D_\Theta}$, with  
 $g(0)=g(\pi)=0$ and $g''\in A^2(D_\Theta)$. Then $Y$ is dense in $\tilde X_D$.
\end{corollary}

\begin{proof}
    Let $f\in \tilde X_D$. Using the density of polynomials in $A^2(D_\theta)$
(see, for instance, Duren and Schuster \cite[p.14]{duren2024bergman})
it follows that for every $\varepsilon>0$ there exists a polynomial $q_\varepsilon$
such that
\begin{equation}\label{prima_aproximare}
\| q_\varepsilon -  s^\frac{\delta}{2}(\pi-s)^\frac{\delta}{2} f\|_{A^2(D_\Theta)}\leqslant \frac{\varepsilon}{2}.
\end{equation}
On the other hand, from \Cref{thm:main} it follows that there exists $g_\varepsilon\in Y$ such that
\begin{equation}\label{doua_aproximare}
\|s^{\frac{\delta}{2}} (\pi-s)^{\frac{\delta}{2}} g_\varepsilon -q_\varepsilon\|_{A^2 (D_\Theta)}\leqslant \frac{\varepsilon}{2}.
\end{equation}
Putting together \eqref{prima_aproximare} and \eqref{doua_aproximare} it follows that 
there exists $g_\varepsilon\in Y$ with
\begin{equation}\label{treia_aproximare}
\|g_\varepsilon - f\|_{A^2_\delta(D_\Theta)}\leqslant \varepsilon,
\end{equation}
so that the conclusion follows.
\end{proof}

For $\varphi\in(0,\pi/2)$ we define the sector
\begin{equation}\label{def_sector}
\Sigma_\varphi := \{\lambda \in \mathbb{C}^*: |\arg \lambda| < \pi/2 + \varphi \}.
\end{equation}
The main ingredient of the proof of \Cref{th_hol_tilde} consists in estimating the resolvents $(\lambda \mathbb{I}-\tilde A_D)^{-1}$
of the operator $\widetilde A_D$ introduced in \eqref{def_dom_A}, \eqref{def_op_A}. As shown in the proof below of \Cref{th_hol_tilde}
these resolvents are extentions to $\tilde X_D$ of the  family of operators
$\left(\widetilde{\mathcal{R}}_\lambda\right)_{\lambda\in \Sigma_\varphi}$, where for each $\lambda\in \Sigma_\varphi$,
the operator $\widetilde{\mathcal{R}}_\lambda$ is defined on the space  $Y$ introduced in \Cref{cor_imp} by
\begin{multline}\label{eq:Rdef}
(\widetilde{\mathcal{R}}_\lambda f)(s)
=
\frac{\sinh(\sqrt{\lambda} (\pi-s))}{\sqrt{\lambda} \sinh(\sqrt{\lambda} \pi)}\int_{0}^{s}\sinh(\sqrt{\lambda}w)f(w)\, {\rm d}w
\\
+
\frac{\sinh(\sqrt{\lambda}s)}{\sqrt{\lambda} \sinh(\sqrt{\lambda} \pi)}
\int_{s}^{\pi}\sinh(\sqrt{\lambda} (\pi-w))f(w)\, {\rm d}w \quad\qquad(s\in D_\Theta,\lambda\in 
\Sigma_\varphi, f\in Y),
\end{multline}
where $\int_0^s$ stands for the integral along the segment  $[0,s]$ and $\int_s^\pi$ for the integral along $[s,\pi]$. More precisely, the following result 
holds.

\begin{theorem}\label{est_princ_princ}
   Let $\delta\in (0,1)$, $\Theta\in \left(0,\frac{\pi}{4}\right)$ and $\varphi\in\left(0,\frac{\pi}{2}-2\Theta\right)$. Then for every $\lambda\in \Sigma_\varphi$, where $\Sigma_\varphi$ has been defined in \eqref{def_sector}, the operator 
  $\widetilde{\mathcal{R}}_\lambda$ can be extended to an operator
  in $\mathcal{L}(\widetilde X_D)$. Moreover, there exists a constant $M>0$ such that
  \begin{equation}\label{estsec_tilde}
      \|\widetilde{\mathcal{R}}_\lambda\|_{\mathcal{L}(\widetilde X_D)}\leqslant \frac{M}{1+|\lambda|}
      \qquad\qquad(\lambda\in \Sigma_\varphi).
  \end{equation}
\end{theorem}

The proof of the above theorem, although elementary, requires an important amount of intermediate estimates, so we postpone it to Appendix I.

We are now in a position to prove \Cref{th_hol_tilde}.

\begin{proof}[Proof of \Cref{th_hol_tilde}] 
Consider $f\in \widetilde X_D$ and let $(f_n)$ be a sequence in $Y$ , where $Y$ has been introduced in \Cref{cor_imp},
such that $f_n\to f$ in $\tilde X_D$. Let  $\varphi\in\left(0,\frac{\pi}{2}-2\Theta\right)$. For each $\lambda\in \Sigma_\varphi$ and $n\in \mathbb{N}$ we set $g_{\lambda,n}=\mathcal{R}_\lambda f_n$, where $\mathcal{R}_\lambda$ are the operators defined in \eqref{eq:Rdef}. It is easy to check that $g_{\lambda,n}\in Y\subset \mathcal{D}(\tilde A_D)$ and that
 \begin{equation}\label{eq_sir_rez}
    \lambda g_{\lambda,n}(s)- g_{\lambda,n}''(s)=f_n(s) \qquad\qquad(n\in \mathbb{N},\lambda\in 
    \Sigma_\varphi,s\in D_\Theta),
    \end{equation}
    \begin{equation}\label{ci_sir_rez}
    g_{\lambda,n}(0)=g_{\lambda,n}(\pi)=0 \qquad\qquad(n\in \mathbb{N}).
    \end{equation}
   Moreover, we know from \Cref{est_princ_princ} that for each $\lambda\in \Sigma_\varphi$,
   $(g_{\lambda,n})$ is a Cauchy sequence in $\tilde X_D$. Combining this fact and \eqref{eq_sir_rez} it follows that
   \begin{equation}\label{dubla_conv}
   g_{\lambda,n}\to g_\lambda \qquad {\rm and}\qquad g_{\lambda,n}''\to f- \lambda  g_\lambda
   \qquad {\rm in}\qquad \tilde X_D.
\end{equation}
On the other hand $g_{\lambda,n}\to g_\lambda$ in $\mathcal{D}'(D_\theta)$ so that $g_{\lambda,n}''\to g_{\lambda,}''$
in $\mathcal{D}'(D_\Theta)$. Combining this fact and \eqref{dubla_conv} it follows that 
\begin{equation}\label{reg_noua}
g_\lambda,\ g_{\lambda}''\in \tilde X_D,    
\end{equation}
\begin{equation}\label{eq_lim_rez}
    \lambda g_{\lambda}(s)- g_{\lambda}''(s)=f(s) \qquad\qquad(\lambda\in 
    \Sigma_\varphi,s\in D_\Theta).
\end{equation}
Additionally, it follows from \eqref{dubla_conv} and \Cref{cor_lem_intors} that
 \begin{equation*}
   g_{\lambda,n}\to g_\lambda \qquad {\rm and}\qquad \frac{{\rm d}^2 g_{\lambda,n}}{{\rm d}x^2}\to f- \lambda  g_\lambda
   \qquad {\rm in}\qquad W^{-1,2}(0,\pi).
\end{equation*}
The above convergences and \eqref{ci_sir_rez} yield that $g_{\lambda,n}\to g_\lambda$ in $W^{1,2}(0,\pi)$ so that 
$g_\lambda(0)=g_\lambda(\pi)=0$, in the sense of \Cref{rem_urma_noua}. This fact, combined with \eqref{reg_noua} and \eqref{eq_lim_rez}
imply that $g_\lambda\in \mathcal{D}(\tilde A_D)$ and
\[
\lambda g_\lambda- \tilde A_D g_\lambda =f.
\]
We have thus shown that for each $\varphi\in\left(0,\frac{\pi}{2}-2\Theta\right)$ we have that $\Sigma_\varphi$ is contained in the resolvent set $\rho(\tilde A_D)$ of $A_D$ and 
\begin{equation}\label{rez_frumoasa}
(\lambda \mathbb{I}-\tilde A_D)^{-1}=\widetilde{\mathcal{R}}_\lambda \qquad\qquad(\lambda\in \Sigma_\varphi).
\end{equation}
On the other hand, we note that since $Y \subset \mathcal{D}(\widetilde A_D)$, we can apply \Cref{cor_imp} to conclude 
that $\mathcal{D}(\widetilde A_D)$ is dense in $\widetilde X_D$. This fact, combined with \eqref{rez_frumoasa} and 
\Cref{est_princ_princ}, shows that $\widetilde A_D$ is sectorial on $\widetilde X_D$ so the first conclusion follows.

Moreover, we have seen in  \Cref{cor_lem_intors} that $\tilde X_D$ is a subspace of $X_D=W^{-1,2}(0,\pi)$. Let
 $\mathbb{T}^D$ be the analytic semigroup on $X_D$ generated by the Dirichlet Laplacian (i.e., the operator $A_D$
 introduced in \Cref{ex_heat_det}). It is clear that 
 \begin{equation}\label{rez_eg}
 \left[(\lambda\mathbb{I}-A_D)^{-1} f\right](x) = \left[(\lambda\mathbb{I}-\tilde A_D)^{-1} f\right](x)
 \qquad\qquad\left(f\in Y,x\in (0,\pi)\right),
 \end{equation}
 since both quantities in the above formula are given by the right hand side of \eqref{eq:Rdef}.
 Taking into account that, according to \Cref{cor_imp}, $Y$ is dense in $\tilde X_D$, it follows that \cref{rez_eg} holds for any $f$ in $\tilde X_D$.  From  \cite[Proposition 4.4]{EN06} we deduce that  $\tilde{\mathbb{T}}^D$ is indeed
 the restriction of $\mathbb{T}^D$ to $\tilde X_D$.

\end{proof}

Let us now give the proof of our main result.

\begin{proof}[Proof of \Cref{th_main}]
The  idea of the proof is to first apply \Cref{prop_reg_stoch}, with an appropriate choice of spaces and operators. We specify below these spaces and operators and we next check that they satisfy all the assumptions in  \Cref{prop_reg_stoch}.

We have seen in \Cref{ex_heat_stoh}   that the process $\psi$ defined by \eqref{duh_stoh_dir} has a  representative which is continuous in time with values in $W^{-1,2}(0,\pi)$
and that it solves \eqref{PROBHEAT1_stoch}.  On the other hand, we have seen in \Cref{th_hol_tilde} that the restriction to $\tilde X_D=A_\delta^2(D_\Theta)$, denoted $\tilde{\mathbb{T}}^D$,  of the semigroup $\mathbb{T}^D$ introduced in \Cref{ex_heat_det}  is an analytic semigroup on $\tilde X_D$. Moreover, let $\Phi=(\Phi_t^D)_{t\geqslant 0}$ be the family of maps defined by \eqref{Phi_t} with $B=B_D$ (defined in \eqref{def_B_primul}) and $\mathbb{T}=\mathbb{T}^D$. As recalled in  \Cref{ex_heat_det} , $\Sigma_D=\begin{bmatrix} \mathbb{T}^D & \Phi^D\end{bmatrix}$ is a well-posed linear control system and
its reachable space $R_{\Sigma_D}$ is well defined and satisfies   $R_{\Sigma_D}=A^2\left(D_\frac{\pi}{4}\right)$, which is continuously embedded in $\tilde X_D$.
Finally, we know from \Cref{expint} that the operators $(L_\tau^D)_{\tau>0}$ defined by \eqref{def_tilde_phi_dir}
are Hilbert-Schmidt  from $L^2\left([0,\tau];\mathbb{C}^2\right)$ to $A^2_\delta(D_\Theta)$.

We are thus in a position to apply \Cref{prop_reg_stoch} to obtain that for $\psi_0\in \tilde X_D$ the process $\psi$ defined by \eqref{duh_stoh_dir} has a continuous representative with values in $\tilde X_D$. 

To show that the above conclusion is sharp, we note that if we have  either that $\delta=0$ and $\Theta\in \left(0,\frac{\pi}{4}\right)$ or that $\delta\in (0,1)$ and $\Theta=\frac{\pi}{4}$, then, according to \Cref{prop_sharp}, the input map $\Phi_\tau^D$ defined in \eqref{DEFFITAUSIMPLU} is not
Hilbert-Schmidt  from $L^2\left([0,\tau];\mathbb{C}^2\right)$ into $A^2_\delta(D_\Theta)$.  
Thus, according to \Cref{th_reg_min}, for $\delta$ and $\Theta$ taking the critical values above,
the mild solution solution $\psi$ of \eqref{PROBHEAT1_stoch}  does not generally take values in $A_\delta^2\left(D_\frac{\pi}{4}\right)$.
\end{proof}

\section{Proof of \Cref{th_main_neu}}\label{sen_meu_main}

In this section we describe the adaptation of our approach for the  heat  equation on $[0,\pi]$ with noise in the Neumann boundary conditions. More precisely, we consider the system
\eqref{PROBHEAT1_stoch_neu}.

As in the case of Dirichlet boundary conditions, our aim consists of proving that the state 
trajectory $\psi$ solving \eqref{PROBHEAT1_stoch_neu} has a continuous in time representative which takes 
values in a Hilbert space of holomorphic functions of weighted Bergman type. Continuing the analogy with the Dirichlet case, this property is closely 
related to regularity results which have been recently obtained for the deterministic counterpart of
\eqref{PROBHEAT1_stoch_neu}, i.e., for the deterministic initial and boundary value problem \eqref{PROBHEAT1_neu}.

The strategy we use to prove the above result is to deduce the essential properties of the semigroup and the input maps associated 
to \eqref{PROBHEAT1_stoch_neu} from the corresponding properties the semigroup and input maps associated to the system with noise in the Dirichlet boundary conditions. To achieve this goal, we need more properties of Bergman type spaces on a rhombus. Although these properties are quite elementary, we did not find them precisely stated in the existing literature, so that we prove them in Appendix II in   \Cref{app2_fin}.

Let $\tilde X_N=A_\delta^{1,2}(D_\Theta)$, where the space $A_\delta^{1,2}(D_\Theta)$ has been defined in 
\eqref{Sob_Berg}.
Let $\tilde A_N:\mathcal{D}(\tilde A_N)\to \tilde X_N$ be the operator defined by 
\begin{equation}\label{def_dom_A_neu}
\mathcal{D}(\widetilde A_N)=\left\{f\in \widetilde X_N\ \ | \ \ f''\in \widetilde X_N,\    f'(0)=f'(\pi)=0\right\},  
\end{equation}
\begin{equation}\label{def_op_A_neu}
\widetilde A_N f = f'' \qquad\qquad(f\in \mathcal{D}(\widetilde A_N)).    
\end{equation}

The main ingredient of the proof of \Cref{th_main_neu} is the following result:

\begin{theorem}\label{th_hol_tilde_neu}
The operator $\widetilde A_N$ defined in \eqref{def_dom_A_neu}, \eqref{def_op_A_neu} generates an analytic
semigroup  $\widetilde{\mathbb{T}}^N$ on $\widetilde X_N$. Moreover, $\widetilde{\mathbb{T}}^N$ is the restriction
to $\widetilde X_N$ of the semigroup $\mathbb{T}^N$ introduced in \Cref{ex_heat_det_neu}. 
\end{theorem}

The proof of the above result requires some preparation.

 \begin{proposition}\label{de_combinat}
Let $\delta\in (0,1)$ and $\Theta\in \left(0,\frac{\pi}{4}\right)$. Then the restriction to $(0,\pi)$ of any function in $A_\delta^{1,2}(D_\Theta)$ lies in $L^2[0,\pi]$. Moreover,
 there exists a constant $K>0$ depending only on $\Theta$ and on $\delta$ such that
 \begin{equation}\label{estimare_mare}
     \|f\|_{L^2[0,\pi]} \leqslant K \|f\|_{A_\delta^{1,2}(D_\Theta)} \qquad\qquad(f\in A_\delta^{1,2}(D_\Theta)). 
\end{equation}
 \end{proposition}
 
The proof of the above result can be easily completed using \Cref{lem_intoarsa} and the following elementary result:

\begin{lemma}\label{lema_noua}
Let $0<\delta<1$. Then there exists a constant $C_\delta>0$ such that
\[
 \int_0^\pi |f(x)|^2\, {\rm d}x
 \le C_\delta \int_0^\pi x^{1+\delta}(\pi-x)^{1+\delta}\bigl(|f(x)|^2+|f'(x)|^2\bigr)\, {\rm d}x
\]
for every differentiable function $f:(0,\pi)\to\mathbb C$ for which the integral on the right-hand side is finite.
\end{lemma}

\begin{proof}
Set $\alpha=1+\delta\in(1,2)$. The claim follows by combining the standard
weighted Hardy estimate on $(0,\pi/2)$ with weight $x^\alpha$ and the analogous
estimate near $\pi$, obtained by the change of variables $y=\pi-x$.
Since this reduction is straightforward, we omit the proof; see
Opic--Kufner \cite{OpicKufner1990} and Kufner--Persson \cite[Section~5.1]{KufnerPersson2003}.
\end{proof}

An important ingredient of the proof of \Cref{th_hol_tilde_neu} is:

\begin{proposition}\label{dens_meu_prop}
The space $\mathcal{D}(\widetilde A_N)$ defined in \eqref{def_dom_A_neu} is dense in $\tilde X_N$.
\end{proposition}

\begin{proof}
Let $Y$ be the space introduced in \Cref{cor_imp}, which means that
 $Y$ is the vector space formed of all the functions $g$ which are holomorphic in $D_\Theta$ and
 continuous on $\overline{D_\Theta}$, with $\frac{{\rm d}^2 g}{{\rm d}s^2}\in A^2(D_\Theta)$ and 
 $g(0)=g(\pi)=0$.

According to \Cref{cor_imp}, for every $f\in \tilde X_N$ there exists a sequence $ (g_k)\subset Y$ such that
\begin{equation}\label{prima_convergenta}
\left\|f'- g_k\right\|_{A_\delta^2(D_\Theta)}\to 0.
\end{equation}
For each $k\in \mathbb{N}$ we denote by $f_k$ the unique function holomorphic on $D_\Theta$ satisfying
\begin{equation}\label{DEFK}
f_k'(s)=  g_k(s) \qquad\qquad(s\in D_\Theta),
\end{equation}
\begin{equation}\label{cu_integrala}
\int_0^\pi f_k(x)\, {\rm d}x=\int_0^\pi f(x)\, {\rm d}x.
\end{equation}
Since $g_k\in A_\delta^2(D_\Theta)$ for every $k\in \mathbb{N}$, using \eqref{DEFK} and \Cref{prop:embedding-DTheta} from Appendix II it follows that 
\begin{equation}\label{este_in}
f_k \in \tilde X_N \qquad\qquad(k\in \mathbb{N}).    
\end{equation}
On the other hand,  from $g_k\in Y$ we have
$g_k''\in A_\delta^2(D_\Theta)$ so that, applying again \Cref{prop:embedding-DTheta} it follows that $g_k'\in A_\delta^2(D_\Theta)$  
Moreover, from \eqref{DEFK} it follows that $f_k''=g_k'$ on $D_\Theta$ for all $k\in \mathbb{N}$, so that.
\begin{equation}\label{este_in_bis}
f_k'' \in \tilde X_N \qquad\qquad(k\in \mathbb{N}).    
\end{equation}
We also note that \eqref{DEFK} implies that
\begin{equation}\label{este_in_margine}
f_k'(0)=f_k'(\pi)=0 \qquad\qquad(k\in \mathbb{N}).    
\end{equation}
Putting together \eqref{este_in}-\eqref{este_in_margine} it follows that
\begin{equation}\label{este_in_domeniu}
f_k \in \mathcal{D}(\tilde A_N) \qquad\qquad(k\in \mathbb{N}).    
\end{equation}
Finally, putting together \eqref{prima_convergenta}, \eqref{cu_integrala} and \Cref{prop:poincare-Qtheta}
it follows that
\[
\|f_k-f\|_{A_\delta^{1,2}(D_\Theta)} \to 0,
\]
which ends the proof of the claimed density property.
\end{proof}

We are now ready to prove \Cref{th_hol_tilde_neu}.

\begin{proof}[Proof of \Cref{th_hol_tilde_neu}]
Let $\Sigma_\varphi$ be the subset of the complex plane introduced in \eqref{def_sector}. We know from the proof of \Cref{th_hol_tilde} that 
for every $\varphi\in\left(0,\frac{\pi}{2}-2\Theta\right)$ we have that
$\Sigma_\varphi\subset \rho(\tilde A_D)$ and
that there exists $M>0$ with
\begin{equation}\label{est_rez_dir}
      \left\|(\lambda I-\widetilde A_D)^{-1}\right\|_{\mathcal{L}(A_\delta^2(D_\Theta))}\leqslant \frac{M}{1+|\lambda|}
      \qquad\qquad(\lambda\in \Sigma_\varphi).
  \end{equation}
  Given $f\in \tilde X_N$ and $\lambda\in \Sigma_\varphi$, we note that, by \eqref{est_rez_dir},  any function $\psi_\lambda$ 
  with
  \begin{equation}\label{der_e_buna}
   \psi_\lambda'(s)=\left[(\lambda I-\widetilde A_D)^{-1} f'\right](s) \qquad\qquad(s\in D_\Theta).   
  \end{equation}
satisfies
\begin{equation}\label{este_der}
 \|\psi_\lambda'\|_{A_\delta^2(D_\Theta)}\leqslant   \frac{M}{1+|\lambda|} \|f'\|_{A_\delta^2(D_\Theta)}
      \qquad\qquad(f\in A^{1,2}_\delta(D_\Theta),\ \lambda\in \Sigma_\varphi).  
\end{equation}
Moreover, we know from \Cref{prop:embedding-DTheta} that every function $\psi_\lambda$ satisfying \eqref{der_e_buna} lies in $A^2_\delta(D_\Theta)$.  
We next select $\psi_\lambda=\psi_\lambda(f)$ such that
\begin{equation}\label{int_e_buna}
  \lambda \int_0^\pi \psi_\lambda(x,0)\, {\rm d}x = \int_0 ^\pi f(x,0)\, {\rm d}x.    
  \end{equation}
Note that the right hand side of the above formula makes sense since, due to \Cref{de_combinat}, both $\psi_\lambda(\cdot,0)$ and $f$
are in $L^2[0,\pi]$.

On the other hand, from \eqref{der_e_buna} it follows that
\begin{equation}\label{ec_pisi_l}
\lambda \frac{\partial\psi_\lambda}{\partial x}(x,0)-\frac{\partial^3\psi_\lambda}{\partial x^3}(x,0)=\frac{\partial f}{\partial x}(x,0) \qquad\qquad (\lambda\in \Sigma_\varphi, x\in (0,\pi)) ,   
\end{equation}
\begin{equation}\label{cl_pisi_l}
\frac{\partial \psi_\lambda}{\partial x}(0,0)=\frac{\partial \psi_\lambda}{\partial x}(0,\pi)=0 \qquad\qquad (\lambda\in \Sigma_\varphi).    
\end{equation}
From \eqref{ec_pisi_l}  it follows that for every $\lambda\in \Sigma_\varphi$ there exists $c_\lambda\in \mathbb{C}$ 
\begin{equation}\label{holds_with}
\lambda \psi_\lambda(x,0)- \frac{\partial^2\psi_\lambda}{\partial x^2}(x,0)=f(x,0)+c_\lambda \qquad\qquad ( x\in (0,\pi)) .   
\end{equation}
Integrating the above formula on $[0,\pi]$ and using \eqref{int_e_buna}, \eqref{cl_pisi_l} it follows that $c_\lambda =0$.
We can thus use \eqref{cl_pisi_l} to obtain that
\begin{equation}\label{618}
\|\psi_\lambda(\cdot,0)\|_{L^2[0,\pi]} \leqslant \frac{K}{1+|\lambda|} \|f(\cdot,0)\|_{L^2[0,\pi]}
\qquad\qquad(\lambda\in \Sigma_\varphi, f\in \tilde X_N),
\end{equation}
where $K$ is a constant depending only on $\varphi$. Combining next \eqref{este_der} and \eqref{618} it follows that
\begin{equation*}
\|\psi_\lambda'\|_{A_\delta^2(D_\Theta)}+\|\psi_\lambda(\cdot,0)\|_{L^2[0,\pi]} \leqslant  \frac{M}{1+|\lambda|} \left(\|f'\|_{A_\delta^2(D_\Theta)}+\|f(\cdot,0)\|_{L^2[0,\pi]}\right),
\end{equation*}
where $K$ is a constant depending only on $\varphi$.
Combining the above estimate with \eqref{618} and \Cref{rem_fonala_foarte} it follows that
\begin{equation}\label{aproape_finala}
\|\psi_\lambda(f)\|_{A_\delta^{1,2}(D_\Theta)} \leqslant  \frac{M}{1+|\lambda|} \|f\|_{A_\delta^{1,2}(D_\Theta)} \qquad\qquad(f\in A_\delta^{1,2}(D_\Theta)).
\end{equation}
Moreover, using the fact that \eqref{ec_pisi_l} holds with $c_\lambda=0$, together with \eqref{cl_pisi_l} and the analiticity 
on $D_\Theta$ of $\psi_\lambda(f)$ and $f$, implies that $\Sigma_\varphi\subset \rho(\tilde A_N)$, $\psi_\lambda\in \mathcal{D}(\tilde A_N)$
and that
\begin{equation}\label{aprticular}
(\lambda \mathbb{I}-\tilde A_N)^{-1} f=\psi_\lambda(f) \qquad\qquad(\lambda\in \Sigma_\varphi,f\in \tilde X_N).
\end{equation}
Thus, using \eqref{aproape_finala}, it follows that 
\begin{equation*}
\|(\lambda\mathbb{I}-\tilde A_N)^{-1}\|_{\mathcal{L}(A_\delta^{1,2}(D_\Theta))} \leqslant  
\frac{M}{1+|\lambda|}  \qquad\qquad(\lambda\in  \Sigma_\varphi),
\end{equation*}
so that the operator $\tilde A_N$ is sectorial. Combining this fact
and \Cref{dens_meu_prop} implies that $\tilde A_N$ generates an analytic semigroup on $\tilde X_N$.

Moreover, we have seen in  \Cref{de_combinat} that $\tilde X_N= A_{\delta}^{1,2}(D_\Theta)$ is a subspace of $X_N=L^{2}[0,\pi]$. Let
 $\mathbb{T}^N$ be the analytic semigroup on $X_N$ generated by the Neumann Laplacian (i.e., the operator $A_N$
 introduced in \Cref{ex_heat_det_neu}). It is clear from \eqref{aprticular}, \eqref{cl_pisi_l} and the fact that \eqref{holds_with} holds with $c_\lambda=0$ that 
\begin{equation*}
 \left[(\lambda\mathbb{I}-A_N)^{-1} f\right](x) = \left[(\lambda\mathbb{I}-\tilde A_N)^{-1} f\right](x)
 \quad\qquad\left(f\in \tilde X_N,x\in (0,\pi)\right).
 \end{equation*}
 Using again  \cite[Proposition 4.4]{EN06} we deduce that  $\tilde{\mathbb{T}}^N$ is indeed
 the restriction of $\mathbb{T}^N$ to $\tilde X_N$.

\end{proof}

We next give the analogue of \Cref{expint} in the case of Neumann boundary conditions, which states as follows:

\begin{proposition}\label{expint_neu}
Let $\tau>0$, $\delta\in (0,1)$ and $0<\Theta<\frac{\pi}{4}$. 
Let $\left(\Phi_{\tau}^{N}\right)_{\tau\geqslant 0}$  be the input maps defined by
\begin{equation}\label{def_input_neu}
\Phi_{\tau}^{N} u = z(\tau,\cdot) \qquad\qquad(\tau\geqslant 0,\ u\in L^2([0,\infty);\mathbb{C}^2)),
\end{equation}
where $z$ satisfies \eqref{PROBHEAT1_neu}.
Then for every $\gamma \in \left[0,\frac{\delta}{4}\right)$ the input map $\Phi_\tau^N$defined in \eqref{def_input_neu} satisfies
\begin{equation}\label{conclu_N}
\sum_{n\in \mathbb{Z}} \left\|\Phi_\tau^N \begin{bmatrix}{\rm e}_{n,\tau,\gamma}\\ 
0\end{bmatrix}\right\|^2_{A^{1,2}_\delta(D_\Theta)}
+  \sum_{n\in \mathbb{Z}} \left\|\Phi_\tau^N\begin{bmatrix} 0 \\ 
{\rm e}_{n,\tau,\gamma}\end{bmatrix}\right\|^2_{A^{1,2}_\delta(D_\Theta)} < \infty ,
\end{equation}
where the functions ${\rm e}_{n,\tau,\gamma}$ have been defined in \eqref{dexp_timp}.
In particular, $\Phi_\tau^N$ is a Hilbert-Schmidt operator from $L^2\left([0,\tau];\mathbb{C}^2\right)$ to $A^{1,2}_\delta(D_\Theta)$.
\end{proposition}

\begin{proof} We first recall from \Cref{ex_heat_det_neu}  that  
\begin{equation}\label{reach_heat_1D_M}
{\rm Ran}\, \Phi_\tau^N  = A^{1,2}(D_\frac{\pi}{4}):=\left\{\eta\in A^2\left(D_\frac{\pi}{4}\right)\ \ \left|\ \ \eta'\in A^2\left(D_\frac{\pi}{4}\right)\right.\right\}.
\end{equation}
Moreover, we have that 
\begin{equation}\label{inca_un_label}
\left(\Phi_\tau^N \begin{bmatrix}
    u_0\\ u_\pi
\end{bmatrix} \right)'(s)= \left(\Phi_\tau^D \begin{bmatrix}
    u_0\\ u_\pi
\end{bmatrix} \right)(s) \qquad\qquad(u_0,u_\pi \in L^2[0,\tau],\ s\in D_\frac{\pi}{4}),
\end{equation}
where $\Phi_\tau^D$ is the input map introduced in \eqref{DEFFITAUSIMPLU}. Indeed, the above formula obviously holds on $(0,\pi)$ thus, by analiticity, for every $s\in D_\frac{\pi}{4}$.

Combining \eqref{inca_un_label}  and \Cref{expint} we obtain that
\begin{equation}\label{HS_complex}
\sum_{n\in \mathbb{Z}} \left\|\left(\Phi_\tau^N\begin{bmatrix}{\rm e}_{n,\tau,\gamma}\\ 
0\end{bmatrix}\right)'\right\|^2_{A^2_\delta(D_\Theta)}
+  \sum_{n\in \mathbb{Z}} \left\|\left(\Phi_\tau^N \begin{bmatrix} 0 \\ 
{\rm e}_{n,\tau,\gamma}\end{bmatrix}\right)'\right\|^2_{A^2_\delta(D_\Theta)} < \infty .
\end{equation}
On the other hand, it is known  (see, for instance, \cite{DAP_Z_Main}) that $\Phi_{\tau}^{N}$ is a Hilbert-Schmidt operator from $L^2([0,\infty);\mathbb{C}^2)$ to
$L^2[0,\pi]$ so that 
\begin{equation*}
\sum_{n\in \mathbb{Z}} \left\|{\Phi_\tau^N}\begin{bmatrix}{\rm e}_{n,\tau,\gamma}\\ 
0\end{bmatrix}\right\|^2_{L^2[0,\pi]}
+  \sum_{n\in \mathbb{Z}} \left\|\Phi_\tau^N \begin{bmatrix} 0 \\ 
{\rm e}_{n,\tau,\gamma}\end{bmatrix}\right\|_{L^2[0,\pi]}^2 < \infty .
\end{equation*}
Putting together the above estimate, \eqref{HS_complex} and \Cref{rem_fonala_foarte} we obtain the conclusion \eqref{conclu_N}. 
\end{proof}

\begin{remark}\label{rem_pentru_n}
We note that if either  $\delta=0$ and $\Theta\in \left(0,\frac{\pi}{4}\right)$
or if $\delta\in (0,1)$ and $\Theta=\frac{\pi}{4}$ then the input map $\Phi_\tau^N$  is not
Hilbert-Schmidt  from $L^2\left([0,\tau];\mathbb{C}^2\right)$ into $A^{1,2}_\delta(D_\Theta)$.    
Indeed, by combining \Cref{prop_sharp} and \eqref{inca_un_label} it follows that, under each of the above assumptions on $\delta$ and $\Theta$ we have
\begin{equation*}
\sum_{n\in \mathbb{Z}} \left\|\left(\Phi_\tau^N\begin{bmatrix}{\rm e}_{n,\tau,0}\\ 
0\end{bmatrix}\right)'\right\|^2_{A^2_\delta(D_\Theta)}
+  \sum_{n\in \mathbb{Z}} \left\|\left(\Phi_\tau^N \begin{bmatrix} 0 \\ 
{\rm e}_{n,\tau,0}\end{bmatrix}\right)'\right\|^2_{A^2_\delta(D_\Theta)} = + \infty .
\end{equation*}
\end{remark}

We are now in a position to give the proof announced in the title of the section.

\begin{proof}[Proof of \Cref{th_main_neu}]
As we  have seen in \Cref{rem_neu_stoch}, the system  \eqref{ec_stoh_abs}, \eqref{init_stoh_abs} has a unique mild solution
defined by \eqref{duh_stoh}, with $A=A_N$ and $B$ given by \eqref{def_B_primul_neumann}. Using \Cref{th_hol_tilde_neu}  it follows that the restriction to $\tilde X_N$ of the semigroup $\mathbb{T}^N$, denoted by $\widetilde{\mathbb{T}}^N$, is an analytic semigroup on $\tilde X_N$. Moreover, we have seen in \Cref{expint_neu} that
the corresponding input maps satisfy the conclusion of \Cref{expint_neu}.  By applying
\Cref{prop_reg_stoch} we can now conclude that $\psi$ has a version continuous in time with values in $A_\delta^{1,2}\left(D_\Theta\right)$.

Finally, by combining \Cref{th_reg_min} and \Cref{rem_pentru_n} it follows that 
for \ $\Theta\in \left(0,\frac{\pi}{4}\right)$ \ the solution $\psi$
does not generally take values in $A^{1,2}(D_\Theta)$ 
  and for $\delta\in (0,1)$ it does not take values in $A_\delta^{1,2}(D_\frac{\pi}{4})$.
\end{proof}

\section{Appendix I: A weighted Bergman space observation on a rhombus}\label{appendixI}

This Appendix is devoted to the proof of \Cref{est_princ_princ}. Given $\Theta\in\left(0,\frac{\pi}{4}\right)$ and $\varphi\in\left(0,\frac{\pi}{2}-2\Theta\right)$, we recall that the rhombus $D_\Theta$ has been defined in \eqref{rtheta} and the sector $\Sigma_\varphi$ has been introduced in \eqref{def_sector}. The following results describe some properties of the elements of $\Sigma_\varphi$ which will be useful in the proof of \Cref{est_princ_princ}. 

\begin{lemma}\label{lem:more}
Let $\lambda=|\lambda|{\rm e}^{i\nu}\in \Sigma_\varphi$ and $\theta\in(-\Theta,\Theta)$. We have that
\begin{enumerate}[label=(\alph*),leftmargin=2em]
\item The function $r\in[0,r_\theta]\to \left| \pi-r{\rm e}^{i\theta}\right|$ is nonincreasing, where $r_\theta=\frac{\pi \sin\Theta}{\sin(\Theta+|\theta|)}$.

\item There exists a constant $c_{\Theta,\varphi}\in\left(0,\frac{\pi}{2}\right)$ such that 
\begin{equation}\label{eq:consig}
\left|\frac{\nu}{2}+\theta\right|\leqslant  \frac{\pi}{2}-c_{\Theta,\varphi},
\end{equation}
and 
\begin{equation}\label{eq:inrel}
\Re \left(\sqrt{\lambda}\left(\pi-r{\rm e}^{i\theta}\right)\right)\geq 0\qquad (r\in [0,r_\theta]),
\end{equation}
where $\sqrt{\lambda}$ denotes the principal branch of the square root function in $\Sigma_\varphi$.
\end{enumerate}
\end{lemma}
\begin{proof}For each $r\in[0,r_\theta]$, let
\[
h(r)=\left| \pi-r{\rm e}^{i\theta}\right|^2=r^2-2\pi r \cos\theta+\pi^2,
\]
and remark that 
\begin{align*}
h'(r)&=2r-2\pi\cos\theta\leq 2r_\theta -2\pi\cos\theta=-\frac{2\pi\sin(|\theta|)}{\sin(\Theta+|\theta|)}\cos(\Theta+|\theta|)\leq 0,
\end{align*}
where we have taken into account that $\Theta+|\theta|\in \left( 0,\frac{\pi}{2}\right)$. Consequently, the function $h$ is nonincreasing in $[0,r_\theta]$ and the first part of the Lemma is proved.

On the other hand, from the fact that $\lambda\in \Sigma_\varphi$, we have that $|\nu|<\frac{\pi}{2}+\varphi$ and  
\[
\left|\frac{\nu}{2}+\theta\right|\leq \frac{\pi}{4}+\frac{\varphi}{2}+\left|\theta\right|\leq \frac{\pi}{4}+\frac{\varphi}{2}+\Theta=\frac{\pi}{2}-\left(\frac{\pi}{4}-\frac{\varphi}{2}-\Theta\right).
\]
Since $\frac{\varphi}{2}\in \left(0,\frac{\pi}{4}-\Theta\right)$, if follows that  \eqref{eq:consig} holds with $c_{\Theta,\varphi}=\frac{\pi}{4}-\frac{\varphi}{2}-\Theta\in \left(0,\frac{\pi}{4}-\Theta\right)$. 

To show \eqref{eq:inrel} we remark that
\[
\Re \left(\sqrt{\lambda}\left(\pi-r{\rm e}^{i\theta}\right)\right)=\sqrt{|\lambda|}\pi \cos\left(\frac{\nu}{2}\right)-\sqrt{|\lambda|}r \cos\left(\frac{\nu}{2}+\theta\right).
\]
By using \eqref{eq:consig}, it follows that
\begin{align*}
\Re \left(\sqrt{\lambda}\left(\pi-r{\rm e}^{i\theta}\right)\right)&\geq \sqrt{|\lambda|} \left(\pi\cos\left(\frac{\nu}{2}\right)-r_\theta \cos\left(\frac{\nu}{2}+\theta\right)\right)\\
&=\frac{\pi \sqrt{|\lambda|} \sin|\theta|}{\sin\left(\Theta+|\theta|\right)}\cos\left(\frac{\nu}{2}-\mbox{sgn}(\theta) \Theta\right)\geq 0,
\end{align*}
where, for the last inequality, we have used the facts that $|\theta|$, $\Theta+|\theta|$ and $
\left|\frac{\nu}{2}-\mbox{sgn}(\theta) \Theta\right|$ belong to  $\left(0,\frac{\pi}{2}\right)$. The proof of the Lemma is now complete.
\end{proof}

\begin{lemma}[Sectorial control of $\Re\sqrt{\lambda}$]\label{lem:sector}
There exists $c_\varphi>0$ such that for all
$\lambda\in\Sigma_\varphi$, 
\[
\Re\sqrt{\lambda}\ge c_\varphi|\sqrt{\lambda} |.
\]
\end{lemma}

\begin{proof}
Write $\lambda=|\lambda|{\rm e}^{i\nu}$ with $|\lambda|>0$ and $|\nu|<\frac{\pi}{2}+\varphi$.
Then $\sqrt{\lambda} =\sqrt{|\lambda|}\,{\rm e}^{i\frac{\nu}{2}}$ and \[\left|\frac{\nu}{2}\right|<\frac{\pi}{4}+\frac{\varphi}{2}<\frac{\pi}{2}.\]
Hence $\cos\displaystyle\left(\frac{\nu}{2}\right)\ge c_\varphi:=\cos\left(\frac{\pi}{4}+\frac{\varphi}{2}\right)>0$ and
$\Re\sqrt{\lambda} =|\sqrt{\lambda} |\cos\displaystyle\left(\frac{\nu}{2}\right)\ge c_\varphi|\sqrt{\lambda} |$.
\end{proof}

\begin{lemma}[Elementary $\sinh$ bounds]\label{lem:sinh}
For all $z\in\mathbb C$,
\begin{enumerate}[label=(\alph*),leftmargin=2em]
\item $|\sinh z|\le {\rm e}^{|\Re z|}$,
\item $|\sinh z|\le |z|{\rm e}^{|\Re z|}$.
\end{enumerate}
Moreover, for $\lambda\in\Sigma_\varphi$ with $|\lambda|\ge 1$,
\begin{enumerate}[label=(\alph*),leftmargin=2em, start=3]
\item there exists $c=c(\varphi)>0$ such that
\[
|\sinh(\sqrt{\lambda} \pi)|\ge c\,{\rm e}^{(\Re\sqrt{\lambda} )\pi}.
\]
\end{enumerate}
\end{lemma}

\begin{proof}
(a) $\sinh z=({\rm e}^z-{\rm e}^{-z})/2$ gives
\[
|\sinh z|\le \tfrac12(|{\rm e}^z|+|{\rm e}^{-z}|)=\tfrac12({\rm e}^{\Re z}+{\rm e}^{-\Re z})\le {\rm e}^{|\Re z|}.
\]

(b) Use $\sinh z=z\int_0^1 \cosh(tz)\,{\rm d}t$ and $|\cosh\xi|\le \tfrac12(|{\rm e}^\xi|+|{\rm e}^{-\xi}|)\le {\rm e}^{|\Re\xi|}$:
\[
|\sinh z|\le |z|\int_0^1 {\rm e}^{|\Re(tz)|}\,{\rm d}t \le |z|{\rm e}^{|\Re z|}.
\]

(c) By Lemma~\ref{lem:sector}, $\Re\sqrt{\lambda} \ge c_\varphi|\sqrt{\lambda} |\ge c_\varphi>0$ for $|\lambda|\ge1$.
Then
\[
|\sinh(\sqrt{\lambda} \pi)|\ge \sinh(\Re(\sqrt{\lambda} \pi))
\ge \tfrac12(1-{\rm e}^{-2\pi c_\varphi}){\rm e}^{(\Re\sqrt{\lambda} )\pi}.
\]
Set $c=\tfrac12(1-{\rm e}^{-2\pi c_\varphi})$ and the proof of the Lemma is complete.
\end{proof}

Now we have all the ingredients needed to give the announced proof. 

\begin{proof}[Proof of \Cref{est_princ_princ}] 
Denote
\begin{multline}\label{prima_jumate}
(\widetilde{\mathcal R}_{0,\lambda} f)(s)
=
\frac{\sinh(\sqrt{\lambda} (\pi-s))}{\sqrt{\lambda} \sinh(\sqrt{\lambda} \pi)}\int_{0}^{s}\sinh(\sqrt{\lambda}w)\,f(w)\, {\rm d}w \\
=\int_0^s I_{0,\lambda}(s,w)f(w)\, {\rm d}w \qquad\qquad(f\in Y, s\in D_\Theta),
\end{multline}
where
\begin{equation}\label{IOO}
I_{0,\lambda}(s,w)=\frac{\sinh\left(\sqrt{\lambda}w\right)\sinh\left(\sqrt{\lambda} (\pi-s)\right)}{\sqrt{\lambda} \sinh\left(\sqrt{\lambda} \pi\right)}
 \qquad\qquad(s,w\in D_\Theta).
\end{equation}
We recall that the space $Y$ has been introduced in \Cref{cor_imp} and consists of all the functions $f$ which are holomorphic in $D_\Theta$ and continuous on $\overline{D_\Theta}$. 

Firstly, we show that
\begin{equation}\label{eq:estim1}
     \int_{D_\Theta} \rho_\delta(s)\left|(\widetilde{\mathcal R}_{0,\lambda} f)(s)\right|^2\, {\rm d}A(s)\leq \frac{M_0}{(1+|\lambda|)^2} \|f\|^2_{A^2_\delta(D_\Theta)} \qquad \qquad (f\in Y),
\end{equation}
for a positive constant $M_0$ independent of $\lambda$.

To obtain the desired estimate \eqref{eq:estim1} it suffices to show that there exist two positive constants $C_1$ and $C_2$ independent of $\lambda$ such that 
\begin{equation}\label{Schur1}
    \sup_{r {\rm e}^{i\theta}\in D_\Theta} \int_0^{r} \left|I_{0,\lambda}(r {\rm e}^{i\theta},t {\rm e}^{i\theta})\right|\widehat{\rho}_\delta(r,t,\theta)
     \, {\rm d}t 
    \leqslant \frac{C_1}{1+|\lambda|} \quad\qquad(\lambda\in \Sigma_\varphi),
\end{equation}
\begin{equation}\label{Schur2}
    \sup_{t {\rm e}^{i\theta}\in D_\Theta} \int_t^{r_\theta} \left|I_{0,\lambda}(r {\rm e}^{i\theta},t {\rm e}^{i\theta})\right| \widehat{\rho}_\delta(r,t,\theta)
     \, {\rm d}r
    \leqslant \frac{C_2}{1+|\lambda |} \quad\qquad(\lambda\in \Sigma_\varphi),
\end{equation}
where $\widehat{\rho}_\delta(r,t,\theta):=\sqrt{\frac{r \rho_\delta(r {\rm e}^{i\theta})}{t \rho_\delta(t {\rm e}^{i\theta})}}$. Indeed, by Cauchy-Schwarz
\begin{multline*}
    \int_{D_\Theta} \rho_\delta(s)\left|(\widetilde{\mathcal R}_{0,\lambda} f)(s)\right|^2\, {\rm d}A(s)=
     \int_{D_\Theta} \rho_\delta(s)\left|\int_0^s I_{0,\lambda}(s,w)f(w)\, {\rm d}w\right|^2\, {\rm d}A(s)\\ =
     \int_{D_\Theta} \rho_\delta(s)\left|\int_0^{|s|} I_{0,\lambda}\left(s,t\frac{s}{|s|}\right)f\left(t\frac{s}{|s|}\right) \frac{s}{|s|}\, {\rm d}t\right|^2\, {\rm d}A(s)\\
     = \int_{-\Theta}^\Theta\int_0^{r_\theta} \rho_\delta(r {\rm e}^{i\theta})\left|\int_0^{r} I_{0,\lambda}(r {\rm e}^{i\theta},t {\rm e}^{i\theta})
     f(t {\rm e}^{i\theta}) {\rm e}^{i\theta} \, {\rm d}t\right|^2 r \, {\rm d}r\, {\rm d}\theta\\
     \leqslant \int_{-\Theta}^\Theta\int_0^{r_\theta} \rho_\delta(r {\rm e}^{i\theta}) \int_0^{r} \left|I_{0,\lambda}(r {\rm e}^{i\theta},t {\rm e}^{i\theta})\right|\frac{1}{\sqrt{t \rho_\delta(t {\rm e}^{i\theta})}}
     \, {\rm d}t \\
     \times \int_0^{r} |I_{0,\lambda}(r {\rm e}^{i\theta},t {\rm e}^{i\theta})| \frac{t r}{\sqrt{t \rho_\delta(t {\rm e}^{i\theta})}} \rho_\delta(t{\rm e}^{i\theta})
     |f(t {\rm e}^{i\theta}) |^2 \, {\rm d}t  \, {\rm d}r\, {\rm d}\theta\\
     \leqslant \frac{C_1}{1+|\lambda |} \int_{-\Theta}^\Theta\int_0^{r_\theta}   \int_0^{r} |I_{0,\lambda}(r {\rm e}^{i\theta},t {\rm e}^{i\theta})|t\, \widehat{\rho}_\delta(r,t,\theta)\rho_\delta(t{\rm e}^{i\theta})
     |f(t {\rm e}^{i\theta}) |^2 \, {\rm d}t  \, {\rm d}r\, {\rm d}\theta\\
     =\frac{C_1}{1+|\lambda |} \int_{-\Theta}^\Theta\int_0^{r_\theta}   \int_t^{r_\theta}  |I_{0,\lambda}(r {\rm e}^{i\theta},t {\rm e}^{i\theta})| t\, \widehat{\rho}_\delta(r,t,\theta)\rho_\delta(t{\rm e}^{i\theta})
     |f(t {\rm e}^{i\theta}) |^2   \, {\rm d}r \, {\rm d}t \, {\rm d}\theta\\
     \leqslant \frac{C_1 C_2}{(1+|\lambda|)^2} \int_{-\Theta}^\Theta\int_0^{r_\theta}  t \rho_\delta(t{\rm e}^{i\theta})
     |f(t {\rm e}^{i\theta}) |^2   \, {\rm d}r \, {\rm d}t \, {\rm d}\theta.
\end{multline*}
It follows that \eqref{eq:estim1} holds for $M_0=C_1C_2$.

Now let us show that \cref{Schur1} and \cref{Schur2} are verified. 

Firstly, we analyze the case $|\lambda|<1$. We begin by remarking that there exist constants $c_1,c_2>0$ such that 
\begin{equation}\label{eticheta}
\left|\sinh (\sqrt{\lambda}s)\right|\leqslant c_1\sqrt{|\lambda|}|s|\qquad\qquad(\lambda,s\in\mathbb{C},\,\, |\lambda|\leqslant 1,\,\,|s|\leqslant \pi),
\end{equation}
\begin{equation}\label{eticheta2}c_2\sqrt{|\lambda|} \leqslant \left|\sinh (\pi\sqrt{\lambda})\right| \qquad\qquad(\lambda\in\Sigma_\varphi,\,\, |\lambda|\leqslant 1).\end{equation}
Indeed, inequality \eqref{eticheta} follows from the boundedness of the entire function $\frac{\sinh (\zeta) }{\zeta}$ in the ball $|\zeta|\leq \pi$. The second inequality \eqref{eticheta2} is a consequence of the fact that $\frac{\sinh (\pi \zeta) }{\pi \zeta}$ does not vanish in the compact set $\left\{\zeta\in\mathbb{C}\,:\, |\arg \zeta|\leq \frac{\pi}{4}+\frac{\varphi}{2}\right\}$.
By taking into account \eqref{IOO}, from inequalities \eqref{eticheta} and \eqref{eticheta2}, we deduce that 
\[
 \left|I_{0,\lambda}(r {\rm e}^{i\theta},t {\rm e}^{i\theta})\right|\widehat{\rho}_\delta(r,t,\theta)\leqslant \frac{c_1^2}{c_2} \left|\pi-r{\rm e}^{i\theta}\right| r^{\frac{1+\delta}{2}}t^{\frac{1-\delta}{2}}\left|\frac{\pi-r{\rm e}^{i\theta}}{\pi-t{\rm e}^{i\theta}}\right|^{\frac{\delta}{2}}.
\]
Since $\delta\in (0,1)$ and since, according to  part {\it (a)} of \Cref{lem:more}, \( \left|\pi-r{\rm e}^{i\theta}\right|\leq \left|\pi-t{\rm e}^{i\theta}\right|\) for $0\leqslant t\leqslant r$, the above estimate implies that \cref{Schur1} and \cref{Schur2} are verified for $|\lambda|<1$.

Let us now consider the case $\lambda=|\lambda|{\rm e}^{i\nu}\in \Sigma_\varphi$, $|\lambda|\geq 1$. We begin by evaluating $I_{0,\lambda}$. For $\theta\in(-\Theta,\Theta)$ and $r,t\in[0,r_\theta]$, by using {\it (a)} and {\it (c)} from Lemma \ref{lem:sinh} we obtain that 
\begin{align*}
\left|I_{0,\lambda}(r {\rm e}^{i\theta},t {\rm e}^{i\theta})\right| &\leq \frac{{\rm e}^{\left|\Re\left(\sqrt{\lambda}\left(\pi-r{\rm e}^{i\theta}\right)\right)\right|}}{c \sqrt{|\lambda|}{\rm e}^{\pi\Re \sqrt{\lambda}}}
\left|\sinh\left(\sqrt{\lambda} t {\rm e}^{i\theta}\right)\right|.
\end{align*}
By taking into account {\it (b)} from Lemma \ref{lem:more}, from the above estimate it follows that
\begin{equation}\label{eq:estkernel}
\left|I_{0,\lambda}(r {\rm e}^{i\theta},t {\rm e}^{i\theta})\right| \leqslant 
\frac{{\rm e}^{ -\sqrt{|\lambda|}\, r \, \varsigma\ }}{c \sqrt{|\lambda|}}
\left|\sinh\left(\sqrt{\lambda} t {\rm e}^{i\theta}\right)\right|,
\end{equation}
where $\varsigma:=\cos\left(\frac{\nu}{2}+\theta\right).$
We pass to the proof of \cref{Schur1}. Firstly, we remark that from {\it (a)} of Lemma \ref{lem:more} and \eqref{eq:estkernel}, we have that 
\begin{multline*}
 \int_0^{r} \left|I_{0,\lambda}(r {\rm e}^{i\theta},t {\rm e}^{i\theta})\right|\widehat{\rho}_\delta(r ,t,\theta)
     \, {\rm d}t \leqslant r^\frac{1+\delta}{2} \int_0^{r} \left|I_{0,\lambda}(r {\rm e}^{i\theta},t {\rm e}^{i\theta})\right|t^{-\frac{1+\delta}{2}}\,{\rm d}t\\ 
     \leqslant \frac{1}{c \sqrt{|\lambda|}} r^\frac{1+\delta}{2} {\rm e}^{-\sqrt{|\lambda|}\, r \varsigma}\int_0^{r} 
\left|\sinh\left(\sqrt{\lambda} t {\rm e}^{i\theta}\right)\right| t^{-\frac{1+\delta}{2}}\,{\rm d}t\\
=\frac{1}{c \sqrt{|\lambda|}} r^\frac{1+\delta}{2} {\rm e}^{-\sqrt{|\lambda|}\, r \varsigma} \int_0^{\frac{1}{\sqrt{|\lambda|}}} 
\left|\sinh\left(\sqrt{\lambda} t {\rm e}^{i\theta}\right)\right| t^{-\frac{1+\delta}{2}}\,{\rm d}t\\+\frac{1}{c \sqrt{|\lambda|}} r^\frac{1+\delta}{2} {\rm e}^{-\sqrt{|\lambda|}\, r \varsigma}\int_{\frac{1}{\sqrt{|\lambda|}}} ^r
\left|\sinh\left(\sqrt{\lambda} t {\rm e}^{i\theta}\right)\right| t^{-\frac{1+\delta}{2}}\,{\rm d}t:=I_1+I_2.
\end{multline*}
We evaluate each of the two terms from above. By using {\it (b)} from Lemma \ref{lem:sinh} have that
\begin{align*}
I_1&\leqslant \frac{{\rm e}}{c} r^\frac{1+\delta}{2} {\rm e}^{-\sqrt{|\lambda|}\, r \varsigma} \int_0^{\frac{1}{\sqrt{|\lambda|}}} 
 t^{\frac{1-\delta}{2}}\,{\rm d}t=\frac{2{\rm e}}{c(3-\delta)} r^\frac{1+\delta}{2} {\rm e}^{-\sqrt{|\lambda|}\, r \varsigma}|\lambda|^{-\frac{3-\delta}{4}}\\&=\frac{2{\rm e}}{c(3-\delta)|\lambda| \varsigma^\frac{1+\delta}{2}} \left(\sqrt{|\lambda|}r \varsigma\right)^\frac{1+\delta}{2} {\rm e}^{-\sqrt{|\lambda|}\, r \varsigma}.
\end{align*}
Since the function $u\in [0,\infty)\to u^\frac{1+\delta}{2} {\rm e}^{-u }$ is bounded and, according to \eqref{eq:consig},  \begin{equation}\label{eq:conute}\varsigma\geq \sin\left(c_{\Theta,\varphi}\right)>0,\end{equation}we deduce that there exists an absolute constant  $C_{1,1}>0$ such that 
\begin{equation}\label{eq:I1}
I_1\leq \frac{C_{1,1}}{|\lambda|}.
\end{equation}

On the other hand, {\it (a)} from Lemma \ref{lem:sinh} implies that 
\begin{align*}
I_2&\leqslant \frac{ r^\frac{1+\delta}{2}}{c \sqrt{|\lambda|}} {\rm e}^{ -\sqrt{|\lambda|}\, r \varsigma}\int_{\frac{1}{\sqrt{|\lambda|}}} ^r 
{\rm e}^{\sqrt{|\lambda|}\, t \varsigma}  t^{-\frac{1+\delta}{2}}\,{\rm d}t\\&=\frac{\left(r \sqrt{|\lambda|}\varsigma\right)^\frac{1+\delta}{2}}{c |\lambda| \varsigma}  {\rm e}^{-\sqrt{|\lambda|}\, r \varsigma}\int_{\varsigma} ^{r \sqrt{|\lambda|} \varsigma} 
{\rm e}^{u}  u^{-\frac{1+\delta}{2}} \,{\rm d}u.
\end{align*}
From the above estimate, \eqref{eq:conute} and the  fact that the function
\[
  v \mapsto v^\frac{1+\delta}{2} {\rm e}^{-v }\int_{\sin(c_{\Theta,\varphi})}^v 
  {\rm e}^{u}  u^{-\frac{1+\delta}{2}} \,{\rm d}u,
\]
is  bounded on $\left[\sin(c_{\Theta,\varphi}),\infty\right)$, we deduce that there exists an absolute constant  $C_{1,2}>0$ such that 
\begin{equation}\label{eq:I2}
I_2\leq \frac{C_{1,2}}{|\lambda|}.
\end{equation}
From \eqref{eq:I1} and \eqref{eq:I2} we deduce that \eqref{Schur1} is verified for $|\lambda|\geq 1$ with $C_1=2\left(C_{1,1}+C_{1,2}\right)$.

Now, we pass to show that \eqref{Schur2} holds true for $\lambda=|\lambda|{\rm e}^{i\nu}\in \Sigma_\varphi$, $|\lambda|\geq 1$. Note that from {\it (a)} of Lemma \ref{lem:more} and \eqref{eq:estkernel}, we get
\begin{multline*}
I:=\int_t^{r_\theta} \left|I_{0,\lambda}(r {\rm e}^{i\theta},t {\rm e}^{i\theta})\right| \widehat{\rho}_\delta(r,t,\theta)
     \, {\rm d}r\leqslant 
t^{-\frac{1+\delta}{2}} \int_t^{r_\theta} \left|I_{0,\lambda}(r {\rm e}^{i\theta},t {\rm e}^{i\theta})\right| r^\frac{1+\delta}{2}
     \, {\rm d}r\\
\leqslant t^{-\frac{1+\delta}{2}}  \int_t^{r_\theta} \frac{{\rm e}^{-\sqrt{|\lambda|}\, r \varsigma}}{c \sqrt{|\lambda|}}
\left|\sinh\left(\sqrt{\lambda} t {\rm e}^{i\theta}\right)\right| r^\frac{1+\delta}{2}  \, {\rm d}r\\
\leqslant \frac{t^{-\frac{1+\delta}{2}}}{c \sqrt{|\lambda|}}  \left|\sinh\left(\sqrt{\lambda} t {\rm e}^{i\theta}\right)\right|  \int_{\sqrt{|\lambda|}\, t \varsigma}^{\infty} {\rm e}^{ - u}
\frac{u^\frac{1+\delta}{2}}{\left(\sqrt{|\lambda|} \varsigma\right)^{\frac{3+\delta}{2}}}  \, {\rm d}u.
\end{multline*}
Firstly we analyze the case $0\leqslant \sqrt{|\lambda|}\, t \leqslant 1$. By using  {\it (b)} of  Lemma \ref{lem:sinh}  we have that
\begin{align*}
t^{-\frac{1+\delta}{2}}\left|\sinh\left(\sqrt{\lambda} t {\rm e}^{i\theta}\right)\right|\leqslant {\rm e} \sqrt{\lambda} t^{\frac{1-\delta}{2}},
\end{align*}
hence, we get
\begin{align*}
I\leqslant \frac{{\rm e}}{c |\lambda| \cos^2\left(\frac{\nu}{2}+ \theta\right)} \left(\sqrt{|\lambda|}\, t \cos\left(\frac{\nu}{2}+ \theta\right)\right)^{\frac{1-\delta}{2}} \int_{\sqrt{|\lambda|}\, t \varsigma}^{\infty} {\rm e}^{ - u}
u^\frac{1+\delta}{2} \, {\rm d} u.
\end{align*}
Since the function $h:[0,1]\rightarrow \mathbb{R}$, defined as
\begin{align*}
h(v)=v^{\frac{1-\delta}{2}}\int_v^\infty {\rm e}^{-u} u^\frac{1+\delta}{2} \, {\rm d}u,
\end{align*}
is bounded, by taking into account \eqref{eq:conute}, we get that there exists a positive constant $C_{2,1}$ such that  
\begin{equation}\label{eq:I3}
I\leqslant \frac{C_{2,1}}{|\lambda|}.
\end{equation}
Secondly, we consider the case $ \sqrt{|\lambda|}\, t > 1$. By using  {\it (a)} of  Lemma \ref{lem:sinh}  we have that
\begin{multline*}
I\leqslant \frac{t^{-\frac{1+\delta}{2}}}{c \sqrt{|\lambda|}} 
{\rm e}^{ \sqrt{|\lambda|}\, t \varsigma}
\int_{\sqrt{|\lambda|}\, t \varsigma}^{\infty} {\rm e}^{- u}
\frac{u^\frac{1+\delta}{2}}{\left(\sqrt{|\lambda|} \varsigma\right)^{\frac{3+\delta}{2}}} \, {\rm d}u \\
\leqslant \frac{\left(\sqrt{|\lambda|}\, t \varsigma\right)^{-\frac{1+\delta}{2}} {\rm e}^{\sqrt{|\lambda|}\, t \varsigma}}{c |\lambda| \varsigma} 
\int_{\sqrt{|\lambda|}\, t \varsigma}^{\infty} {\rm e}^{-u}
u^\frac{1+\delta}{2} \, {\rm d}u.
\end{multline*}
By noting that the function
\begin{align*}
h(v)=v^{-\frac{1+\delta}{2}} {\rm e}^{v} \int_v^\infty {\rm e}^{-u} u^\frac{1+\delta}{2} \, {\rm d}u,
\end{align*}
is bounded on $\left[\sin(c_{\Theta,\varphi}), \infty\right)$ and taking once more into account \eqref{eq:conute}, we get 
\begin{equation}\label{eq:I4}
I\leqslant \frac{C_{2,2}}{|\lambda|},\end{equation}
for a positive constant $C_{2,2}$ independent of $\lambda$. Finally, from \eqref{eq:I3}-\eqref{eq:I4} we obtain that \eqref{Schur2} holds for $|\lambda|\geq 1$ with $C_2=2 \max\{C_{2,1},C_{2,2}\}$.

To evaluate the second term in $\widetilde{\mathcal R}_\lambda$, we denote
\begin{multline}\label{adoua_jumate}
(\widetilde{\mathcal R}_{\pi,\lambda} f)(s)
=
\frac{\sinh(\sqrt{\lambda}s)}{\sqrt{\lambda} \sinh(\sqrt{\lambda} \pi)}\int_{s}^{\pi}\sinh(\sqrt{\lambda} (\pi-w))\,f(w)\,dw \\
=\int_s^{\pi} I_{\pi,\lambda}(s,w)f(w)\, {\rm d}w \qquad\qquad(f\in A_\delta^2(D_\Theta), s\in D_\Theta),
\end{multline}
where
\begin{equation}\label{Ipipi}
I_{\pi,\lambda}(s,w)=\frac{\sinh(\sqrt{\lambda}s)}{\sqrt{\lambda} \sinh(\sqrt{\lambda} \pi)} \sinh(\sqrt{\lambda} (\pi-w))  \qquad\qquad(s,w\in D_\Theta).
\end{equation}

To estimate $\widetilde{\mathcal R}_{\pi,\lambda}$ it is not difficult to check that for every $f\in Y$ we have
$$
(\widetilde{\mathcal R}_{\pi,\lambda} f)(\pi-s)=(\widetilde{\mathcal R}_{0,\lambda} \tilde f)(s),
$$
where $\tilde f(w)=f(\pi-w)$. Consequently, from the estimate \cref{eq:estim1} already obtained for $\widetilde{\mathcal R}_{0,\lambda}$, we deduce that 
\begin{equation}\label{eq:estim2}
     \int_{D_\Theta} \rho_\delta(s)\left|(\widetilde{\mathcal R}_{\pi,\lambda} f)(s)\right|^2\, {\rm d}A(s)\leq \frac{M_\pi}{(1+|\lambda|)^2} \|f\|^2_{A^2_\delta(D_\Theta)} \quad \qquad (f\in Y),
\end{equation}
for a positive constant independent of $\lambda$.

From \cref{eq:estim1} and \cref{eq:estim2} it follows that 
\begin{equation}\label{eq:estim3}
     \int_{D_\Theta} \rho_\delta(s)\left|(\widetilde{\mathcal R}_{\lambda} f)(s)\right|^2\, {\rm d}A(s)\leq \frac{2(M_0+M_\pi)}{(1+|\lambda|)^2} \|f\|^2_{A^2_\delta(D_\Theta)} \quad \quad (f\in Y).
\end{equation}
The conclusion of the theorem is a consequence of inequality \cref{eq:estim3} and of the density of $Y$ in $\widetilde{X}_D$. 
\end{proof}
\section{Appendix II: Weighted Bergman and Bergman--Sobolev estimates on a rhombus}\label{app2_fin}

Let $0<\Theta<\frac{\pi}{4}$, and consider the rhombus $D_\Theta$ defined in \eqref{rtheta}. $D_\Theta$ can be equivalently described as
\begin{equation}\label{eq:DTheta-cartesian}
D_\Theta
=
\{x+iy\in\mathbb C: 0<x<\pi,\ |y|<a(x)\},
\quad
 a(x):=\tan\Theta\,\min(x,\pi-x).
\end{equation}
For $0<\delta<1$, we have introduced in \eqref{def_rho_delta} the weight  \( \rho_\delta(s)=|s|^\delta|\pi-s|^\delta\) defined for $s\in D_\Theta$. 
The weighted Bergman space \(A_\delta^2(D_\Theta)\) and the weighted Bergman--Sobolev space \(
A_\delta^{1,2}(D_\Theta)\) have been introduced in \eqref{fara_rho} and \eqref{Sob_Berg}, respectively.

In this Appendix we record two basic facts. The first is the property that a function $f\in{\rm Hol(D_\Theta)}$ with $f'\in A_\delta^2(D_\Theta)$ also verifies that $f\in A_\delta^2(D_\Theta)$. The second is a Poincar\'e-type inequality for functions in \( A_\delta^{1,2}(D_\Theta)\).

\begin{proposition}\label{prop:embedding-DTheta}
Let $f$ be holomorphic in $D_\Theta$ and assume that
\(f'\in A_\delta^2(D_\Theta)\). Then
\(f\in A_\delta^2(D_\Theta).\)
\end{proposition}

\begin{proof}
Since $D_\Theta$ is a rhombus, there exist $s_0\in\mathbb C$ and two $\mathbb R$-linearly independent complex numbers $a,b$ such that
\[
D_\Theta=T(Q),
\qquad
Q:=(-1,1)^2,
\qquad
T(x,y)=s_0+ax+by,
\]
and we may choose $T$ so that \(T(-1,-1)=0\) and \(T(1,1)=\pi\). 
Thus $T$ is a real affine bijection from $Q$ onto $D_\Theta$, with constant Jacobian
\(
J:=|\det \nabla T|>0.
\)

Define
\[
F(x,y):=f(T(x,y)) \qquad\qquad(x,y \in (-1,1)).
\]
Let $s=T(x,y)$. Since $f$ is holomorphic, differentiation in the real variables gives
\[
F_x(x,y)=a\,f'(s)=a\,f'(T(x,y)),
\qquad
F_y(x,y)=b\,f'(s)=b\,f'(T(x,y)).
\]
In particular, $F\in C^1(Q)$.

Set
\[
\sigma(x,y):=(x+y+2)(2-x-y)=4-(x+y)^2.
\]
Because $T$ is bi-Lipschitz and sends $\begin{bmatrix} -1\\ -1\end{bmatrix}$ to $\begin{bmatrix} 0\\ 0\end{bmatrix}$ and $\begin{bmatrix} 1\\ 1\end{bmatrix}$ to 
$\begin{bmatrix} \pi \\ 0\end{bmatrix}$, we have
\[
|T(x,y)| \asymp \sqrt{(x+1)^2+(y+1)^2},\qquad |T(x,y)-\pi|\asymp \sqrt{(1-x)^2+(1-y)^2}.
\]
Here and in the sequel, the notation $f\asymp g$ means that $f$ and $g$ are equal up to multiplicative constants on their domain of definition. Since for nonnegative numbers $u,v$ one has 
\[
\frac{u+v}{\sqrt2}\le \sqrt{u^2+v^2}\le u+v,
\]
it follows that
\begin{equation}\label{eq:comppon}
\rho_\delta(T(x,y))
\asymp \sigma(x,y)^\delta,
\qquad (x,y)\in Q.
\end{equation}
Therefore, using the change of variables $s=T(x,y)$,
\[
\iint_Q |F_x(x,y)|^2\sigma(x,y)^\delta\,{\rm d}x\,{\rm d}y<\infty,\quad 
\iint_Q |F_y(x,y)|^2\sigma(x,y)^\delta\,{\rm d}x\,{\rm d}y<\infty.
\]

We claim that
\begin{equation}\label{eq:fxyl2}
\iint_Q |F(x,y)|^2\sigma(x,y)^\delta\,{\rm d}x\,{\rm d}y<\infty.
\end{equation}
First note that on the vertical strip $|x|\le \tfrac12$ we have $|x+y|\le \tfrac32$, hence
\[
\sigma(x,y)=4-(x+y)^2\ge 4-\frac94=\frac74.
\]
Thus 
\(
F_y\in L^2\Bigl(\bigl(-\tfrac12,\tfrac12\bigr)\times(-1,1)\Bigr),
\) 
and by Fubini's theorem, there exists some $x_0\in\left(-\tfrac12,\tfrac12\right)$ such that
\[
\int_{-1}^1 |F_y(x_0,y)|^2\,{\rm d}y<\infty.
\]
For $y\in(-1,1)$,
\[
F(x_0,y)=F(x_0,0)+\int_0^y F_y(x_0,t)\,{\rm d}t.
\]
By Cauchy--Schwarz,
\[
\int_{-1}^1 |F(x_0,y)|^2\,{\rm d}y 
\le 4|F(x_0,0)|^2+8\int_{-1}^1 |F_y(x_0,t)|^2\,{\rm d}t<\infty.
\]

Now fix $y\in(-1,1)$ and $x\in(-1,1)$, and let $I_x$ be the interval with end points $x_0$ and $x$. Then
\[
F(x,y)=F(x_0,y)+\int_{x_0}^x F_x(\xi,y)\,{\rm d}\xi,
\]
so that
\begin{align}\nonumber 
\sigma(x,y)^\delta |F(x,y)|^2
&\le 2\sigma(x,y)^\delta |F(x_0,y)|^2
+2\sigma(x,y)^\delta \left|\int_{x_0}^x F_x(\xi,y)\,{\rm d}\xi\right|^2\\
&\le 2\sigma(x,y)^\delta|F(x_0,y)|^2
+2 K(x,y) \int_{-1}^1 |F_x(\xi,y)|^2\sigma(\xi,y)^\delta\,{\rm d}\xi, \label{eq:jkl}
\end{align}
where $K(x,y)=\sigma(x,y)^\delta\left(\int_{I_x}\sigma(\xi,y)^{-\delta}\, {\rm d}\xi\right)$. 

We now claim that
\begin{equation}\label{eq:jkl2}
K:=\sup_{(x,y)\in Q}K(x,y)<\infty.
\end{equation}
Indeed, with the change of variables $u=x+y$, $u_0=x_0+y$, we have $u\in(-2,2)$ and $u_0\in\left(-\tfrac32,\tfrac32\right)$, while 
\(
\sigma(\xi,y)=4-(\xi+y)^2.
\) 
Thus
\[
\int_{I_x}\sigma(\xi,y)^{-\delta}\,{\rm d}\xi=
\int_{I_u} (4-t^2)^{-\delta}\,{\rm d}t,
\]
where $I_u$ is the interval with endpoints $u_0$ and $u$. Since $0<\delta<1$, we have $(4-t^2)^{-\delta}\in L^1(-2,2)$, so the integral above is uniformly bounded. Moreover, since $\sigma(x,y)^\delta\le 4^\delta$, it follows that  $K<\infty$. Therefore, from \eqref{eq:jkl}
 and \eqref{eq:jkl2} we deduce that
\[
\sigma(x,y)^\delta |F(x,y)|^2
\le
2\sigma(x,y)^\delta |F(x_0,y)|^2
+2K\int_{-1}^1 |F_x(\xi,y)|^2\sigma(\xi,y)^\delta\,{\rm d}\xi.
\]
Integrating in $x$ the above inequality and taking into account that the integral $\int_{-1}^1 \sigma(x,y)^\delta\,{\rm d}x$ is bounded uniformly in $y$, 
we obtain that
\[
\int_{-1}^1 \sigma(x,y)^\delta |F(x,y)|^2\,{\rm d}x
\le
C|F(x_0,y)|^2
+C\int_{-1}^1 |F_x(\xi,y)|^2\sigma(\xi,y)^\delta\,{\rm d}\xi.
\]
Integrating now in $y$, we infer that
\begin{align*}
\iint_Q |F(x,y)|^2\sigma(x,y)^\delta\,{\rm d}x\,{\rm d}y
\le &
C\int_{-1}^1 |F(x_0,y)|^2\,{\rm d}y\\
&+C\iint_Q |F_x(\xi,y)|^2\sigma(\xi,y)^\delta\,{\rm d}\xi\,{\rm d}y
<\infty.
\end{align*}
This proves the claim \eqref{eq:fxyl2}.

Finally, since \eqref{eq:comppon} shows that $\rho_\delta(T(x,y))$ is comparable to $\sigma(x,y)^\delta$ and $T$ has constant Jacobian $J$,
\[
\int_{D_\Theta}|f(s)|^2\rho_\delta(z)\,{\rm d}A(s)
=J\iint_Q |F(x,y)|^2\rho_\delta(T(x,y))\,{\rm d}x\,{\rm d}y
<\infty.
\]
Hence $f\in A_\delta^2(D_\Theta)$.
\end{proof}

Before stating the second result i, this section, we remark that, according to \Cref{de_combinat}, 
$A_\delta^{1,2}(D_\Theta)\subset L^2[0,\pi]$ so that $\int_0^\pi f(x)\,{\rm d}x$
is a well defined quantity. We can thus state the following Poincar\'e type inequality:

\begin{proposition}\label{prop:poincare-Qtheta}
There exists a constant $C>0$, depending only on $\Theta$ and $\delta$, 
such that every $f\in A^{1,2}_\delta(D_\Theta)$  satisfies
\begin{equation}\label{eq:poincaremedie}
\|f\|_{A_\delta^2(D_\Theta)}
\le
C\left(
\|f'\|_{A_\delta^2(D_\Theta)}
+\left|\int_0^\pi f(x)\,{\rm d}x\right|
\right).
\end{equation}
\end{proposition}

\begin{proof}
Write $s=x+iy\in D_\Theta$. Set
\[
\omega(x):=x^\delta(\pi-x)^\delta,
\qquad
\widetilde\omega(x):=a(x)\omega(x),
\qquad 0<x<\pi.
\]
Since for $s\in D_\theta$ we have $|y|<a(x)\le (\tan\Theta)x$ and $|y|<a(x)\le (\tan\Theta)(\pi-x)$, it follows that
\[
x\le |x+iy|\le (\sec\Theta)x,
\qquad
\pi-x\le |\pi-(x+iy)|\le (\sec\Theta)(\pi-x).
\]
Hence,  for $s=x+iy\in D_\Theta$ we have
\begin{equation}\label{eq:weight-comparison}
\rho_\delta(s)\asymp \omega(x),
\end{equation}
with constants depending only on $\Theta$ and $\delta$.

We first reduce the weighted $L^2$ norm on $D_\Theta$ to a weighted trace norm on $(0,\pi)$.  
For $x\in(0,\pi)$ and $|y|<a(x)$, the fundamental theorem of calculus in the vertical direction gives
\[
f(x+iy)=f(x)+\int_0^y \partial_y f(x+it)\,{\rm d}t.
\]
Since $f$ is holomorphic, $\partial_y f=i f'$, and therefore
\[
|f(x+iy)|^2
\le
2|f(x)|^2+2\left|\int_0^y f'(x+it)\,{\rm d}t\right|^2.
\]
Integrating in $y\in(-a(x),a(x))$ and applying Cauchy--Schwarz inequality, we obtain
that for every $x\in (0,\pi)$ we have
\[
\int_{-a(x)}^{a(x)} |f(x+iy)|^2\,{\rm d}y
\le
2a(x)|f(x)|^2
+2a^2(x)\int_{-a(x)}^{a(x)} |f'(x+it)|^2\,{\rm d}t.
\]
Multiplying by $\omega(x)$, integrating in $x$, and using the boundedness of $a$, we get
\begin{align*}
\int_0^\pi\int_{-a(x)}^{a(x)} |f(x+iy)|^2\omega(x)\,{\rm d}y\,{\rm d}x
&\le
C\int_0^\pi |f(x)|^2\widetilde\omega(x)\,{\rm d}x \notag\\
&\quad + C\int_0^\pi\int_{-a(x)}^{a(x)} |f'(x+iy)|^2\omega(x)\,{\rm d}y\,{\rm d}x.
\end{align*}
By \eqref{eq:weight-comparison}, the above inequality  gives
\begin{multline}\label{eq:reduce-to-trace-rho}
\int_{D_\Theta}|f(s)|^2\rho_\delta(s)\, {\rm d}A(s)\\
\le
C\int_0^\pi |f(x)|^2\widetilde\omega(x)\,{\rm d}x
+C\int_{D_\Theta}|f'(s)|^2\rho_\delta(s)\,{\rm d}A(s).
\end{multline}
Thus it remains to estimate the one-dimensional weighted norm of $f$ appearing in the right hand side term of \eqref{eq:reduce-to-trace-rho}. Indeed, we recall that there exists $C>0$, depending only on $\Theta$ and $\delta$, such that for any absolutely continuous function $g$ defined on $(0,\pi)$ we have that 
\begin{equation}\label{eq:trace-poincare}
\int_0^\pi |g(x)|^2\widetilde\omega(x)\,{\rm d}x
\le
C\left(
\int_0^\pi |g'(x)|^2\widetilde\omega(x)\,{\rm d}x
+\left|\int_0^\pi g(x)\,{\rm d}x\right|^2
\right).
\end{equation}
Inequality \eqref{eq:trace-poincare} is a particular case of the one-dimensional weighted Poincar\'e inequality
proved by Chua--Wheeden \cite[Theorem~1.4]{ChuaWheeden2000}, applied with
$a=0$, $b=\pi$, $p=q=2$, $v=dx$, and $\mu=w=\widetilde\omega$.
The required hypothesis reduces to the finiteness of the associated Hardy-type
quantities, which follows immediately from the definition of $\widetilde\omega$. 

Now we are able to complete the proof of  \Cref{prop:poincare-Qtheta}. Since $$\widetilde{\omega}(x)\asymp x^{1+\delta} (\pi-x)^{1+\delta} \qquad (0<x<\pi),$$ we can apply \Cref{lem_intoarsa} to deduce that 
$$\int_0^\pi |f'(x)|^2\widetilde\omega(x)\,{\rm d}x \le
C \|f'\|^2_{A^2_\delta(D_\Theta)}.
$$
Applying  \eqref{eq:trace-poincare} to $g(x)=f(x)$, from the above inequality we obtain that \Cref{eq:poincaremedie} holds and the proof of proposition is complete.
\end{proof}

\begin{remark}\label{rem_fonala_foarte}
 By combining the above proposition and \Cref{de_combinat} it follows that
\[
f\longmapsto \|f'\|_{A_\delta^2(D_\Theta)}+\|f\|_{L^2[0,\pi]}
\]
is a norm on $A_\delta^{1,2}(D_\Theta)$ equivalent to the original norm
\[
f\longmapsto \|f\|_{A_\delta^2(D_\Theta)}+\|f'\|_{A_\delta^2(D_\Theta)}.
\]
\end{remark}

{\bf Acknowledgments.} The authors thank George Weiss and Sylvain Ervedoza for the careful reading of the manuscript and 
for their helpful remarks.

\bibliographystyle{alpha}
\bibliography{biblio}

@misc{ervedoza2025reachablespaceparabolicequations,
      title={On the reachable space for parabolic equations}, 
      author={Sylvain Ervedoza and Adrien Tendani-Soler},
      year={2025},
      eprint={2507.15407},
      archivePrefix={arXiv},
      primaryClass={math.AP},
      url={https://arxiv.org/abs/2507.15407}, 
}

@article{ChuaWheeden2000,
  author  = {Chua, Seng-Kee and Wheeden, Richard L.},
  title   = {Sharp conditions for weighted 1-dimensional Poincar\'e inequalities},
  journal = {Indiana University Mathematics Journal},
  volume  = {49},
  number  = {1},
  year    = {2000},
  pages   = {143--175},
  doi     = {10.1512/iumj.2000.49.1754},
}

@book{KufnerPersson2003,
  author    = {Kufner, Alois and Persson, Lars-Erik},
  title     = {Weighted Inequalities of Hardy Type},
  series    = {Series in Approximations and Decompositions},
  volume    = {31},
  publisher = {World Scientific},
  address   = {River Edge, NJ},
  year      = {2003},
  isbn      = {981-238-324-7}
}

@book{OpicKufner1990,
  author    = {Opic, Bohum{\'\i}r and Kufner, Alois},
  title     = {Hardy-type Inequalities},
  series    = {Pitman Research Notes in Mathematics Series},
  volume    = {219},
  publisher = {Longman Scientific \& Technical},
  address   = {Harlow},
  year      = {1990},
  isbn      = {0-582-04021-0}
}

@article{weiss1989admissibility_con,
  title={Admissibility of unbounded control operators},
  author={Weiss, George},
  journal={SIAM Journal on Control and Optimization},
  volume={27},
  number={3},
  pages={527--545},
  year={1989},
  publisher={SIAM}
}

@article {FatRus2,
    AUTHOR = {Fattorini, H. O. and Russell, D. L.},
     TITLE = {Exact controllability theorems for linear parabolic
             equations in one space dimension},
   JOURNAL = {Arch. Rational Mech. Anal.},
  FJOURNAL = {Archive for Rational Mechanics and Analysis},
    VOLUME = {43},
      YEAR = {1971},
     PAGES = {272--292},
      ISSN = {0003-9527},
MRREVIEWER = {F. M. Kirillova},
}

@article {Seid79,
    AUTHOR = {Seidman, Thomas I.},
     TITLE = {Time-invariance of the reachable set for linear control
              problems},
   JOURNAL = {J. Math. Anal. Appl.},
  FJOURNAL = {Journal of Mathematical Analysis and Applications},
    VOLUME = {72},
      YEAR = {1979},
    NUMBER = {1},
     PAGES = {17--20},
      ISSN = {0022-247X},
       DOI = {10.1016/0022-247X(79)90271-3},
}

@article{fatt78,
  title={Reachable states in boundary control of the heat equation are independent of time},
  author={Fattorini, HO},
  journal={Proceedings of the Royal Society of Edinburgh Section A: Mathematics},
  volume={81},
  number={1-2},
  pages={71--77},
  year={1978},
  publisher={Royal Society of Edinburgh Scotland Foundation}
}

@article{weiss1989admissible,
  title={Admissible observation operators for linear semigroups},
  author={Weiss, George},
  journal={Israel Journal of Mathematics},
  volume={65},
  number={1},
  pages={17--43},
  year={1989},
  publisher={Springer}
}

@book{EN06,
  title={A short course on operator semigroups},
  author={Engel, Klaus-Jochen and Nagel, Rainer},
  year={2006},
  publisher={Springer}
}

@book{duren2024bergman,
  title={Bergman spaces},
  author={Duren, Peter and Schuster, Alexander},
  volume={100},
  year={2024},
  publisher={American Mathematical Society}
}

@article{hartmann2021separation,
  title={Separation of singularities for the {B}ergman space and application to control theory},
  author={Hartmann, Andreas and Orsoni, Marcu-Antone},
  journal={Journal de Math{\'e}matiques Pures et Appliqu{\'e}es},
  volume={150},
  pages={181--201},
  year={2021},
  publisher={Elsevier}
}

@article{fkirine2024evolution,
  title={On evolution equations with white-noise boundary conditions},
  author={Fkirine, Mohamed and Hadd, Said and Rhandi, Abdelaziz},
  journal={Journal of Mathematical Analysis and Applications},
  volume={535},
  number={1},
  pages={128087},
  year={2024},
  publisher={Elsevier}
}

@article {MR1953738,
    AUTHOR = {Al\`os, Elisa and Bonaccorsi, Stefano},
     TITLE = {Stability for stochastic partial differential equations with
              {D}irichlet white-noise boundary conditions},
   JOURNAL = {Infin. Dimens. Anal. Quantum Probab. Relat. Top.},
  FJOURNAL = {Infinite Dimensional Analysis, Quantum Probability and Related
              Topics},
    VOLUME = {5},
      YEAR = {2002},
    NUMBER = {4},
     PAGES = {465--481},
      ISSN = {0219-0257},
       DOI = {10.1142/S0219025702000948},
       URL = {https://doi.org/10.1142/S0219025702000948},
}

@article {MR1899108,
    AUTHOR = {Al\`os, Elisa and Bonaccorsi, Stefano},
     TITLE = {Stochastic partial differential equations with {D}irichlet
              white-noise boundary conditions},
   JOURNAL = {Ann. Inst. H. Poincar\'{e} Probab. Statist.},
  FJOURNAL = {Annales de l'Institut Henri Poincar\'{e}. Probabilit\'{e}s et
              Statistiques},
    VOLUME = {38},
      YEAR = {2002},
    NUMBER = {2},
     PAGES = {125--154},
      ISSN = {0246-0203},
       DOI = {10.1016/S0246-0203(01)01097-4},
       URL = {https://doi.org/10.1016/S0246-0203(01)01097-4},
}

@article {MR3315663,
    AUTHOR = {Brze\'{z}niak, Zdzis\l  and Goldys, Ben and Peszat, Szymon and
              Russo, Francesco},
     TITLE = {Second order {PDE}s with {D}irichlet white noise boundary
              conditions},
   JOURNAL = {J. Evol. Equ.},
  FJOURNAL = {Journal of Evolution Equations},
    VOLUME = {15},
      YEAR = {2015},
    NUMBER = {1},
     PAGES = {1--26},
      ISSN = {1424-3199},
       DOI = {10.1007/s00028-014-0246-2},
       URL = {https://doi.org/10.1007/s00028-014-0246-2},
}

@article {MR4561684,
    AUTHOR = {Goldys, Ben and Peszat, Szymon},
     TITLE = {Linear parabolic equation with {D}irichlet white noise
              boundary conditions},
   JOURNAL = {J. Differential Equations},
  FJOURNAL = {Journal of Differential Equations},
    VOLUME = {362},
      YEAR = {2023},
     PAGES = {382--437},
      ISSN = {0022-0396},
       DOI = {10.1016/j.jde.2023.03.003},
       URL = {https://doi.org/10.1016/j.jde.2023.03.003},
}

@book {MR747979,
    AUTHOR = {Cannon, John Rozier},
     TITLE = {The one-dimensional heat equation},
    SERIES = {Encyclopedia of Mathematics and its Applications},
    VOLUME = {23},
 PUBLISHER = {Addison-Wesley Publishing Company, Advanced Book Program,
              Reading, MA},
      YEAR = {1984},
     PAGES = {xxv+483},
      ISBN = {0-201-13522-1},
       DOI = {10.1017/CBO9781139086967},
       URL = {https://doi.org/10.1017/CBO9781139086967},
}

@article {MR4474841,
    AUTHOR = {Ervedoza, Sylvain and Le Balc'h, K\'{e}vin and Tucsnak, Marius},
     TITLE = {Reachability results for perturbed heat equations},
   JOURNAL = {J. Funct. Anal.},
  FJOURNAL = {Journal of Functional Analysis},
    VOLUME = {283},
      YEAR = {2022},
    NUMBER = {10},
     PAGES = {Paper No. 109666, 61},
      ISSN = {0022-1236},
       DOI = {10.1016/j.jfa.2022.109666},
       URL = {https://doi.org/10.1016/j.jfa.2022.109666},
}

@book {TucsnakWeiss,
    AUTHOR = {Tucsnak, Marius and Weiss, George},
     TITLE = {Observation and control for operator semigroups},
    SERIES = {Birkh\"auser Advanced Texts: Basler Lehrb\"ucher.
              [Birkh\"auser Advanced Texts: Basel Textbooks]},
 PUBLISHER = {Birkh\"auser Verlag, Basel},
      YEAR = {2009},
     PAGES = {xii+483},
      ISBN = {978-3-7643-8993-2},
       URL = {http://dx.doi.org/10.1007/978-3-7643-8994-9},
}

@article {DAP_Z_Main,
    AUTHOR = {Da Prato, G. and Zabczyk, J.},
     TITLE = {Evolution equations with white-noise boundary conditions},
   JOURNAL = {Stochastics Stochastics Rep.},
  FJOURNAL = {Stochastics and Stochastics Reports},
    VOLUME = {42},
      YEAR = {1993},
    NUMBER = {3-4},
     PAGES = {167--182},
      ISSN = {1045-1129},
       DOI = {10.1080/17442509308833817},
       URL = {https://doi-org.docelec.u-bordeaux.fr/10.1080/17442509308833817},
}

@book {DAP_Z_Book,
    AUTHOR = {Da Prato, Giuseppe and Zabczyk, Jerzy},
     TITLE = {Stochastic equations in infinite dimensions},
    SERIES = {Encyclopedia of Mathematics and its Applications},
    VOLUME = {152},
   EDITION = {Second},
 PUBLISHER = {Cambridge University Press, Cambridge},
      YEAR = {2014},
     PAGES = {xviii+493},
      ISBN = {978-1-107-05584-1},
       DOI = {10.1017/CBO9781107295513},
       URL = {https://doi-org.docelec.u-bordeaux.fr/10.1017/CBO9781107295513},
}

@misc{hairer2023,
      title={An {I}ntroduction to {S}tochastic {PDE}s}, 
      author={Martin Hairer},
      year={2023},
      eprint={0907.4178},
      archivePrefix={arXiv},
      primaryClass={math.PR},
      url={https://arxiv.org/abs/0907.4178}, 
}

@article{abreu2013stochastic,
  title={The stochastic {W}eiss conjecture for bounded analytic semigroups},
  author={Abreu, Jamil and Haak, Bernhard and Van Neerven, Jan},
  journal={Journal of the London Mathematical Society},
  volume={88},
  number={1},
  pages={181--201},
  year={2013},
  publisher={Oxford University Press}
}

@article {Saitoh,
    AUTHOR = {Aikawa, Hiroaki and Hayashi, Nakao and Saitoh, Saburou},
     TITLE = {The {B}ergman space on a sector and the heat equation},
   JOURNAL = {Complex Variables Theory Appl.},
  FJOURNAL = {Complex Variables. Theory and Application. An International
              Journal},
    VOLUME = {15},
      YEAR = {1990},
    NUMBER = {1},
     PAGES = {27--36},
      ISSN = {0278-1077},
}

@article {HKT_2020,
    AUTHOR = {Hartmann, Andreas and Kellay, Karim and Tucsnak, Marius},
     TITLE = {From the reachable space of the heat equation to {H}ilbert
              spaces of holomorphic functions},
   JOURNAL = {J. Eur. Math. Soc. (JEMS)},
  FJOURNAL = {Journal of the European Mathematical Society (JEMS)},
    VOLUME = {22},
      YEAR = {2020},
    NUMBER = {10},
     PAGES = {3417--3440},
      ISSN = {1435-9855},
       DOI = {10.4171/jems/989},
       URL = {https://doi.org/10.4171/jems/989},
}

\end{document}